\documentclass[a4paper,reqno,11pt,oneside]{amsart}
\usepackage{amsmath,geometry}
\usepackage{graphicx}
\usepackage{mathrsfs}
\usepackage{amssymb,amsmath, amsfonts, amsthm}
\usepackage{latexsym}
\usepackage{enumitem}
\usepackage{color}
\usepackage[dvips]{epsfig}
\usepackage[colorlinks]{hyperref}
\hypersetup{urlcolor=blue, citecolor=red, linkcolor=blue}
\usepackage{pgfplots}
\usepackage{bm}

\allowdisplaybreaks

\usepackage{cite}
\usepackage{relsize}
\usepackage{esint}
\usepackage{verbatim}
\usepackage{mathrsfs}

\numberwithin{equation}{section}

\newcommand{\e}{\varepsilon}

\newcommand{\R}{\mathbb{R}}
\newcommand{\N}{\mathbb{N}}

\newcommand{\Bn}{{\mathbb{B}^n}}
\newcommand{\dBn}{{\mathbb{S}^{n-1}}}

\newcommand{\Bd}{{\mathbb{B}^2}}
\newcommand{\de}{\partial}

\DeclareMathOperator{\ind}{ind}

\renewcommand{\a}{\alpha}
\renewcommand{\b}{\beta}
\renewcommand{\d}{\delta}
\newcommand{\vfi}{\varphi}

\renewcommand{\l}{\lambda}

\renewcommand{\t}{\theta}
\renewcommand{\O}{\Omega}

\newcommand{\Sph}{{\mathbb{S}}}
\renewcommand{\L}{\Lambda}
\newcommand{\D}{{\mathfrak{D}}}

\newcommand{\dsh}{{2^\sharp}}
\newcommand{\dst}{{2^*}}

\newcommand{\inv}{\mathscr{I}}
\renewcommand{\(}{\left(}
\renewcommand{\)}{\right)}
\DeclareMathOperator{\PV}{P.V.\,}

\newcommand*{\abs}[1]{\left\vert #1\right\vert}

\newtheorem{theorem}{Theorem}[section]
\newtheorem*{theorem*}{Theorem}
\newtheorem{lemma}[theorem]{Lemma}

\newtheorem{definition}[theorem]{Definition}
\newtheorem{proposition}[theorem]{Proposition}

\theoremstyle{definition}

\title[Prescribing nearly constant curvatures on balls]{Prescribing nearly constant curvatures on balls}

\author{Luca Battaglia}
\address{Luca Battaglia, Dipartimento di Matematica e Fisica, Universit\`a degli Studi Roma Tre, Via della Vasca Navale 84, 00146 Roma (Italy)}
\email{luca.battaglia@uniroma3.it}
\author{Sergio Cruz-Blázquez}
\address{Sergio Cruz Blázquez, Departamento de Análisis Matemático, Universidad de Granada, Avenida de Fuente Nueva s/n, 18071 Granada (Spain)}
\email{sergiocruz@ugr.es}
\author{Angela Pistoia}
\address{Angela Pistoia, Dipartimento di Scienze di Base e Applicate per l'Ingegneria, Sapienza Universit\`a di Roma, Via Antonio Scarpa 10, 00161 Roma (Italy)}
\email{angela.pistoia@uniroma1.it}

\begin{document}

\begin{abstract}
In this paper we address two boundary cases of the classical Kazdan-Warner problem. More precisely, we consider the problem of prescribing the Gaussian and boundary geodesic curvature on a disk of $\R^2$, and the scalar and mean curvature on a ball in higher dimensions, via a conformal change of the metric. We deal with the case of negative interior curvature and positive boundary curvature. Using a Ljapunov-Schmidt procedure, we obtain new existence results when the prescribed functions are close to constants.
\end{abstract}\
\date\today
\thanks{S.C. acknowledges financial support from the Spanish Ministry of Universities and Next Generation EU funds, through a \textit{Margarita Salas} grant from the University of Granada, by the FEDER-MINECO Grant PID2021-122122NB-I00 and by J. Andalucia (FQM-116). This work was carried out during his long visit to the University “Sapienza Università di Roma”, to which he is grateful.
The three authors are partially supported by the group GNAMPA of the Istituto Nazionale di Alta Matematica (INdAM). In particular, the first and second author are funded by the project ``Fenomeni di blow-up per equazioni nonlineari'', project code CUP\_E55F22000270001
}
\maketitle
\section{Introduction}
A classical problem arising in Geometric Analysis consists in prescribing certain geometric quantities on Riemannian manifolds via a conformal change of metric. It dates back to \cite{berger71,kazdan-warner74}, where the authors proposed the following problem: \textit{given a smooth function $K$ defined on compact surface $(M^2,g)$, can $K$ be achieved as the Gaussian curvature of $M$ with respect to some conformal metric $\tilde g=e^ug$?}\\
Analytically, this reduces to solve the following equation in $u$:
\begin{equation*}
-\Delta_g u+2k_g=2Ke^u\quad\hbox{in }M,
\end{equation*}
where $\Delta_g$ is the Laplace-Beltrami operator and $k_g$ denotes the Gaussian curvature of $M$ relative to $g$. Over the last few decades, this equation has received significant attention and it is not possible to give here a comprehensive list of references; a collection of results can be found in \cite{aubinbook}.\\

If $M$ has a boundary, then boundary conditions are in order. Here we consider a nonlinear boundary condition corresponding to conformally prescribing the boundary geodesic curvature $H$, for some given function $H$ defined on $\de M$. In this case we are led to the boundary value problem:
\begin{equation}\label{gg}\left\{\begin{array}{ll}
-\Delta_g u+2k_g=2Ke^u &\hbox{in }M,\\
\de_\nu u+2h_g=2He^{\frac u2}&\hbox{on }\de M,
\end{array}\right.
\end{equation}
where $\nu$ is the exterior normal vector to $\de M$ and $h_g$ its initial geodesic curvature. Equation \eqref{gg} has been considered in particular situations; the case of constant $K$ and $H$ has been studied by Brendle in \cite{brendle} by means of a parabolic flow. In this situation, some classification results are available when $M$ is an annulus or the half-space, see \cite{asun12,li-zhu95,zhang}. The case of nonconstant curvatures was first addressed by Cherrier in \cite{cherrier}, but his results are obstructed by the presence of Lagrange multipliers. Recently, the general case with $K<0$ in surfaces topologically different from the disk has been studied in \cite{lsmr22}, and a blow-up analysis has been performed.\\


Generally speaking, the case of a disk is challenging due to the noncompact nature of the conformal map group acting on it, similarly to the Nirenberg problem on $\Sph^2$. The problem becomes
\begin{equation}\label{gg-d}\left\{\begin{array}{ll}
-\Delta u=2Ke^u &\hbox{in }\Bd,\\
\de_\nu u+2=2He^{\frac u2}&\hbox{on }\Sph^1,
\end{array}\right.
\end{equation}
Integrating \eqref{gg-d} and applying the Gauss-Bonnet Theorem, we obtain
\begin{equation*}
\int_\Bd Ke^u+\int_{\Sph^1}He^{\frac u2}=2\pi,
\end{equation*}
from which we see that $K$ or $H$ needs to be somewhere positive. Some partial results are available for the case in which one of the curvatures is zero, see \cite{chang-yang88,chang-liu96,dalio-martinazzi-riviere, guo-liu, li-liu05, liu-huang05}. However, up to our knowledge, there are few results available for the case of nonconstant functions $K$ and $H$. In \cite{CB-Ruiz18}, the problem is posed in a new variational setting and existence of solutions in the form of global minimizers are obtained for nonnegative and symmetric curvatures. Existence results for not necessarily symmetric, nonnegative curvatures are found in \cite{ruiz} via a Leray-Schauder degree argument. In \cite{bmp} the authors construct blowing-up solutions for \eqref{gg-d} under certain nondegeneracy assumptions on $K$ and $H$ using a Ljapunov-Schmidt reduction.\\
Concerning the blow-up behaviour of sequences of solutions, a rather exhaustive study is given in \cite{jlsmr} (see also \cite{dalio-martinazzi-riviere,guo-liu}). In particular, it is shown that if $K<0$, then problem \eqref{gg-d} only admits blow-ups at boundary points where the scaling-invariant function $\D_2:\Sph^1\to\R$ defined as 
\begin{equation}\label{d2}
\D_2=\frac{H}{\sqrt{-K}}
\end{equation}
is greater or equal than one (see \cite[Theorem 1.1]{jlsmr} and \cite[Theorem 1.4]{lsmr22}).\\

The natural analogue of this question in higher dimensions is the prescription of the scalar curvature of a manifold and the mean curvature of the boundary, and has received more attention.\\
More precisely, if $(M^n,g)$ is a Riemannian manifold of dimension $n\ge3$ with boundary, and $K:\overline M\to\R,H:\de M\to\R$ are given smooth functions, it consists of finding positive solution for the boundary problem
\begin{equation}\label{sm}
\left\{\begin{array}{ll}-\frac{4(n-1)}{n-2}\Delta_g u+k_gu=K u^\frac{n+2}{n-2}&\hbox{in }\quad M,\\\frac2{n-2}\de_\nu u+h_g u=Hu^\frac n{n-2}&\hbox{on }\quad\de M.\end{array}\right. 
\end{equation}
Here $k_g$ and $h_g$ denote the scalar and boundary mean curvatures of $M$ with respect to $g$. If $u>0$ solves \eqref{sm}, then the metric $\tilde g=u^\frac4{n-2}g$ satisfies $k_{\tilde g}=K$ and $h_{\tilde g}=H$.\\
In the literature we can find many partial results for this equation. The case of prescribing a scalar flat metric with constant boundary mean curvature is known as the Escobar problem, in strong analogy with the Yamabe problem. Its study was initiated by Escobar in \cite{escobar1,escobar2,escobar-garcia}, with later contributions in \cite{almaraz10,marques07,marques05,Mayer-Ndiaye17}. Different settings with constant curvatures are considered in \cite{brendle-chen14,crs19,escobar3,han-li00,han-li99}. Some results are available for the case of nonconstant functions when one of them is equal to zero. Existence results for the scalar flat problem are given in \cite{aca08,cxy98,dma04,xu-zhang16}, while the works \cite{baema02,baema05,li96} concern the case with minimal boundaries.\\

On the other hand, the problem with nonconstant functions $K$ and $H$ has received comparatively little study. In this regard, we highlight \cite{alm02}, which contains perturbative results about nearly constant positive curvature functions on the unit ball of $\R^n$. 

The case of nonconstant $K>0$ and $H$ of arbitrary sign was also considered in \cite{dma03} in the half sphere of $\R^3$, and a blow-up analysis was carried out. As for negative $K$, in \cite{chs18} the authors study the equation \eqref{sm} with $K<0$ and $H<0$ by means of a geometric flow, in the spirit of \cite{brendle}, but solutions are obtained up to Lagrange multipliers. Finally, in the recent work \cite{cmr22} the case with nonconstant functions $K<0$ and $H$ of arbitrary sign is treated on manifolds of nonpositive Yamabe invariant. Similarly to the two dimensional case, it is shown that the nature of the problem changes greatly depending on whether the function $\D_n:\de M\to\R$ given by 
\begin{equation}\label{dn}
\D_n=\sqrt{n(n-1)}\frac{H}{\sqrt{-K}}
\end{equation} is less than one over the entire boundary or not. When $\D_n<1$ the energy functional becomes coercive and a global minimizer can be found. However, if $\D_n\ge1$ somewhere on $\de M$, a min-max argument and a careful blow-up analysis are needed to recover existence of solutions, although only in dimension three.\\

In this paper, we will focus on the following perturbative version of \eqref{gg-d} and \eqref{sm}:

\begin{equation}\label{pe}\tag{$P^2_\e$}\left\{\begin{array}{ll}
-\Delta u=-2(1+\e K(x))e^u &\hbox{in }\Bd,\\
\de_\nu u+2=2\D_2(1+\e H(x))e^{\frac u2}&\hbox{on }\Sph^1,
\end{array}\right.
\end{equation}
and if $n\ge3$
\begin{equation}\label{pne}\tag{$P^n_\e$}
\left\{\begin{array}{ll}-\frac{4(n-1)}{n-2}\Delta u=-(1+\e K)u^\frac{n+2}{n-2}&\hbox{in }\quad\Bn,\\\frac2{n-2}\de_\nu u+u=\frac{\D_n}{\sqrt{n(n-1)}}(1+\e H)u^\frac n{n-2}&\hbox{on }\quad\dBn,\end{array}\right. 
\end{equation}
where $\D_n>1$ is defined in \eqref{d2} and \eqref{dn}, $K:\mathbb B^n\to\R,H:\mathbb S^{n-1}\to\R$ are smooth with bounded derivatives and the parameter $\e\in\R$ small.\\

Our main result for problem \eqref{pe} reads as follows:
\begin{theorem}\label{th2d} Assume $\D_2\ne\frac2{\sqrt3}$, let $\psi:\Sph^1\to\R$ be defined by
$$\psi(\xi):=\frac{2\pi}{\sqrt{\D_2^2-1}}\(\(\D_2-\sqrt{\D_2^2-1}\)K(\xi)-2\D_2H(\xi)\)$$
and $\Phi_1:\Sph^1\to\R$ be defined by
$$\Phi_1(\xi):=\(\D_2-\sqrt{\D_2^2-1}\)\de_\nu K(\xi)-2\D_2(-\Delta)^\frac12H(\xi)$$
and $\Phi_m$ be defined as in Definition \ref{def-constants}.
If one of the following holds true:
\begin{enumerate}
\item For any global maximum $\xi$ of $\psi$ there exists $m=m(\xi)\ge1$ such that $\Phi_j(\xi)=0>\Phi_m(\xi)$ for any $j<m$;
\item For any global minimum $\xi$ of $\psi$ there exists $m=m(\xi)\ge1$ such that $\Phi_j(\xi)=0<\Phi_m(\xi)$ for any $j<m$;
\item For any critical point $\xi$ of $\psi$ there exists $m=m(\xi)\ge1$ such that $\Phi_j(\xi)=0\ne\Phi_m(\xi)$ for any $j<m$, $\psi$ is Morse and
$$\sum_{\{\xi:\nabla\psi(\xi)=0,\,\Phi_m(\xi)<0\}}(-1)^{\ind_\xi\nabla\psi}\ne1;$$
\end{enumerate}
then, Problem \eqref{pe} has a solution for $|\e|$ small enough.
\end{theorem}\
Our main result for problem \eqref{pne} reads as follows:
\begin{theorem}\label{thn}Let $\psi:\dBn\to\R$ be defined by 
$$\psi(\xi):=\mathtt a(\D_n)K(\xi)-\mathtt b(\D_n)H(\xi),$$ with $\mathtt a(\D_n),\mathtt b(\D_n)$ as in \eqref{ab0}, $\Phi_1:\dBn\to\R$ be defined by
$$\Phi_1(\xi):=\de_\nu K(\xi),$$
and $\Phi_m:\dBn\to\R$ be defined as in Definition \ref{def-constants}.
If one of the following holds true:
\begin{enumerate}
\item For any global maximum $\xi$ of $\psi$ there exists $m=m(\xi)\ge1$ such that $\Phi_j(\xi)=0>\Phi_m(\xi)$ for any $j<m$;
\item For any global minimum $\xi$ of $\psi$ there exists $m=m(\xi)\ge1$ such that $\Phi_j(\xi)=0<\Phi_m(\xi)$ for any $j<m$;
\item For any critical point $\xi$ of $\psi$ there exists $m=m(\xi)\ge1$ such that $\Phi_j(\xi)=0\ne\Phi_m(\xi)$ for any $j<m$, $\psi$ is Morse and
$$\sum_{\{\xi:\nabla\psi(\xi)=0,\,\Phi_m(\xi)<0\}}(-1)^{\ind_\xi\nabla\psi}\ne1;$$
\end{enumerate}
then, Problem \eqref{pe} has a solution for $|\e|$ small enough.
\end{theorem}\

Problems \eqref{pe} and \eqref{pne} share many similarities, not only for their geometric importance, but also from an analytic point of view.\\
In fact, they both have critical terms in the interior and in the boundary nonlinearities: exponential nonlinearities in \eqref{pe} are critical in view of the Moser-Trudinger inequalities, whereas in \eqref{pne} we have the critical Sobolev exponent and the critical trace exponent
\begin{equation}\label{exp}
\frac{n+2}{n-2}=\dst-1,\qquad\frac n{n-2}=\dsh-1.
\end{equation}
Moreover, since we are prescribing a negative curvature in the interior and a positive curvature in the boundary, the two nonlinear terms have different sign and are therefore in \emph{competition}.\\
Theorem \ref{th2d} seems to be the first result of prescribing both nearly constant curvatures on a disk. Similar results were recently obtained in \cite{bcfp} in the case of zero curvature in the interior and in \cite{gp} for the sphere.
Theorem \ref{thn} is the counterpart of the result obtained in \cite{alm02}, where the authors perturb the positive constant curvature on the unit ball of $\R^n$. \\

We also provide higher-order expansions of the reduced energy functional, which permits to consider also some cases of degenerate critical points. This is the case when the functionals $\Phi_m(\xi)$ play a role in Theorems \ref{th2d}, \ref{thn}.\\
Such expansions require sharper estimates (see Proposition \ref{coro-deriv} and Appendix \ref{section-app-A}) and both derivatives of $K,H$ and nonlocal terms appear. In particular, nonlocal terms are present only if the order of the expansion is high enough, depending on the dimension. At the first order, we only get the fractional Laplacian in the two-dimensional case, which is why $\Phi_1(\xi)$ is defined differently in Theorems \ref{th2d} and \ref{thn}.\\
The definition of $\Phi_m$ for $m\ge2$ is rather involved and it is therefore postponed to Definition \ref{def-constants}.\\

Finally, we point out that, in Theorem \ref{th2d}, $\Phi_1$ can be seen as the normal derivative (up to a constant) of the functional $\psi$, which can be naturally extended from the circle to the closed disk. More precisely, for $\xi\in\overline\Bd$, we set
$$\Psi(\xi):=\frac{2\pi}{\sqrt{\D_2^2-1}}\(\D_2-\sqrt{\D_2^2-1}\)K(\xi)-2\D_2\hat H(\xi)\qquad\hbox{where}\quad\left\{\begin{array}{ll}\Delta\hat H=0&\hbox{in }\Bd\\\hat H=H&\hbox{in }\Sph^1\end{array}\right.;$$
therefore, in view of the Dirichlet-to-Neumann characterization of the fractional Laplacian, we have $\Phi_1(\xi)=\de_\nu\Psi(\xi)$ for any $\xi\in\Sph^1$.\\
Quite interestlingly, this fact has no higher-dimensional counterpart in Theorem \ref{thn}.\\
The assumption $\D_2\not=\frac2{\sqrt3}$ (i.e. $\alpha_\D\not=0$ in Proposition \ref{ene2}) allows to apply the degree argument to the function which also depend on the extra parameter that only appears in the 2D case (see (2) of Proposition \ref{ene2}). It would be interesting to understand whether this is a mere technical assumption or not and also whether it has some geometrical meaning.\\

The plan of the paper is as follows.\\
In Section \ref{sec-preli} we introduce some notation and preliminaries which we will use in the following; in Section \ref{sec-energy} we study the energy functional associated to the system and show some of its crucial properties; in Section \ref{sec-linear} we apply the Ljapunov-Schmidt finite dimensional reduction; in Section \ref{sec-critical} we study the existence of critical points to the reduced energy functional; finally, in the Appendix we prove some crucial asymptotic estimates.\\

\section{Notation and Preliminaries}\label{sec-preli}

We remind that $\Bn$ will denote the unit ball of $\mathbb R^n$, for $n\ge2$. We consider the well-known inversion map $\inv:\R^n_+\to\Bn$ defined by

\begin{equation}\label{inv}
\inv(\bar x,x_n)=\(\frac{2\bar x}{\abs{\bar x}^2+(x_n+1)^2},\frac{1-\abs{\bar x}^2-x_n^2}{\abs{\bar x}^2+(x_n+1)^2}\),\quad(\bar x,x_n)\in\R^{n-1}\times\R_+.
\end{equation}
Straightforward computations show that $\inv\circ\inv=\hbox{Id}$, therefore $\inv^{-1}$ has the same expression. For more details about this map, see for instance \cite{escobar-garcia}, Section 2.\\
We point out that, up to the sign of the last coordinate, $\inv$ extends the stereographic projection from $\de\R^n_+$ to $\dBn$ and, in dimension $2$, it coincides with the Riemann map from the half-plane to the disk. In particular, $\inv$ is a conformal map and satisfies
\begin{equation*}
\inv^\star g_{\Bn}=\varrho\abs{dx}^2,\quad\varrho(\bar x,x_n)=\frac{4}{\(\abs{\bar x}^2+(x_n+1)^2\)^2}.
\end{equation*}
For convenience, we define $\rho:\R^n_+\to\R_+$ by 
\begin{equation*}\label{rho}
\rho=\left\{\begin{array}{ll}
\varrho^\frac{n-2}4&\hbox{if}\quad n\ge3,\\
\log\varrho &\hbox{if}\quad n=2.
\end{array}\right.
\end{equation*}\\
We point out that $\rho$ satisfies \eqref{gg} or \eqref{sm} for some particular choices of the curvatures. More precisely, in dimension $n=2$
\begin{equation}\label{rho2}
\left\{\begin{array}{ll}
-\Delta\rho=0&\hbox{in }\quad\R^2_+\\
\de_\nu\rho=2e^\frac\rho2&\hbox{on }\quad\de\R^2_+,
\end{array}\right.
\end{equation}
while in dimensions $n\ge3$
\begin{equation*}
\left\{\begin{array}{ll}-\frac{4(n-1)}{n-2}\Delta\rho=0&\hbox{in }\R^n_+,\\\frac2{n-2}\de_\nu\rho=\rho^\frac n{n-2}&\hbox{on }\de\R^n_+.
\end{array}\right.
\end{equation*}\

For the reader's convenience, we collect here all the constants that appear in our computations. In the following we agree that $\D=\D_n$.
\begin{definition}\label{def-constants} Let $n\ge2$ and $\D>1$. We define
\begin{align*}
\L_n=&\left\{\begin{array}{ll}4&\hbox{if}\quad n=2\\4n(n-1)&\hbox{if}\quad n\ge3\end{array}\right.,\\
\a_n=&\left\{\begin{array}{ll}2&\hbox{if}\quad n=2\\\frac{(n-2)^2}{8n(n-1)}&\textit{if}\quad n\ge3\end{array}\right.,\\
\b_n=&\left\{\begin{array}{ll}2&\hbox{if}\quad n=2\\2\sqrt\frac n{n-1}&\textit{if}\quad n\ge3\end{array}\right.,\\
a_{n,i,j}=&\L_n^\frac n2\int_{\R^n_+}\frac{|\bar y|^{2i}y_n^j}{\(\abs{\bar y}^2+(y_n+\D)^2-1\)^n}d\bar ydy_n,\\
b_{n,j}=&\L_n^\frac{n-1}2\b_n\frac\D{\(\D^2-1\)^\frac{n-2j-1}2}\int_{\de\R^n_+}\frac{|\bar y|^{2j}}{\(\abs{\bar y}^2+1\)^{n-1}}d\bar y,\\
c_{n,m}=&\L_n^\frac{n-1}2\b_n\D\(\D^2-1\)^\frac{m-n+1}2\int_{\de\R^n_+}|\bar y|^m\(\frac1{\(|\bar y|^2+1\)^{n-1}}-\sum_{j=0}^\frac{m-n}2(-1)^j\frac{(n+j-2)!}{j!(n-2)!}\frac1{|\bar y|^{2(n+j-1)}}\)d\bar y,\\
d_{n,j}=&\L_n^\frac n2\omega_{n-2}\int_0^\pi\sin^jtdt\\
e_n=&\L_n^\frac{n-1}2\b_n\D\omega_{n-2}\\
A_{n,i,j}=&\frac1{(j-2i)!(2i)!}\frac{(n-3)!!}{(2i)!!(n+2i-3)!!},\\
B_{n,j}=&\frac{(n-3)!!}{(2j)!(2j)!!(n+2j-3)!!},\\
C_{n,m,i}=&\frac{(m-i-1)!}{(n-1)!i!(m-n-2i)!}(\D^2-1)^i\D^{m-n-2i}\\
D_{n,m}=&\frac{\(\frac{n+m+1}2-2\)!}{\(\frac{m-n+1}2\)!(n-2)!}\(\D^2-1\)^\frac{m-n+1}2\\
\mathsf I_{n,m,i}(K)=&\L_n^\frac n2\int_\Bn\(K(z)-\sum_{|\a|\le m}\frac1{|\a|!}\de_{\xi_\a}K(\xi)(z-\xi)^\a\)\frac{|z+\xi|^{2i}\(1-|z|^2\)^{m-n-2i}}{|z-\xi|^{2m}}dz\\
\mathsf J_{n,m}(H)=&\L_n^\frac{n-1}2\b_n\D\int_{\dBn}\(H(z)-\sum_{|\a|\le m-1}\frac1{|\a|!}\de_{\xi_\a}H(\xi)(z-\xi)^\a\)\frac{|z+\xi|^{m-n+1}}{|z-\xi|^{n+m-1}}dz
\end{align*}
In particular, set
\begin{equation}\label{ab0}
\begin{aligned}&\mathtt a(\D_n):=a_{n,0,0}=\L_n^\frac n2\int_{\R^n_+}\frac{1}{\(\abs{\bar y}^2+(y_n+\D_n)^2-1\)^n}d\bar ydy_n,\\
&\mathtt b(\D_n):=b_{n,0}=\L_n^\frac{n-1}2\b_n\frac{\D_n}{\(\D_n^2-1\)^\frac{n-1}2}\int_{\de\R^n_+}\frac{1}{\(\abs{\bar y}^2+1\)^{n-1}}d\bar y.\end{aligned}
\end{equation}
We define the functionals $\Phi_j(\xi)$ as follows.\\
For $j\le n-2$ we set:
$$\Phi_j(\xi):=\left\{\begin{array}{ll}(1+\xi)^j\(-\sum_{i=0}^{\frac j2}a_{n,i,j}A_{n,i,j}\de_\nu^{j-2i}\Delta_\tau^iK(\xi)+b_{n,j}B_{n,j}\Delta^jK(\xi)\)&\hbox{if }j\hbox{ even}\\(1+\xi)^j\sum_{i=0}^{\frac{j-1}2}a_{n,i,j}A_{n,i,j}\de_\nu^{j-2i}\Delta_\tau^iK(\xi)&\hbox{if }j\hbox{ odd}\\\end{array}\right.$$
For $j=n-1,n$ we set:
$$\Phi_{n-1}(\xi):=\left\{\begin{array}{ll}0&\hbox{if }n\hbox{ even}\\(1+\xi_n)^{n-1}e_nB_{n,\frac{n-1}2}\Delta^\frac{n-1}2H(\xi)&\hbox{if }n\hbox{ odd}\end{array}\right.$$
$$\Phi_n(\xi):=(1+\xi_n)^{n-1}\((-1)^n\sum_{i=0}^{\lfloor\frac{n-1}2\rfloor}a_{n,i,n-1}A_{n,i,n-1}\de_\nu^{n-1-2i}\Delta_\tau^iK(\xi)+\mathsf J_{n,n-1}(H)\).$$
For $m\ge n$ we set:
\begin{align*}
\Phi_{2m-n+1}(\xi):=&(1+\xi_n)^m\(\sum_{i=0}^{\lfloor\frac{m-n}2\rfloor}\sum_{j=0}^{\lfloor\frac m2\rfloor}(-1)^{m-n-i-1}A_{n,j,m}C_{n,m,i}d_{2m-n-2i-2j}\de_\nu^{m-2j}\Delta_\tau^jK(\xi)\right.\\
+&\left.\left\{\begin{array}{ll}0&\hbox{if }n\hbox{ even or }m\hbox{ odd}\\(-1)^\frac{m-n+1}2e_nB_{n,\frac m2}D_{n,m}\Delta^\frac m2H(\xi)&\hbox{if }n\hbox{ odd and }m\hbox{ even}\\\end{array}\right.\)
\end{align*}
\begin{align*}
\Phi_{2m-n+2}(\xi):=&(1+\xi_n)^m\(\sum_{i=0}^{\lfloor\frac{m-n}2\rfloor}(-1)^{m-n-i-1}C_{n,m,i}\mathsf I_{n,m,i}(K)\right.\\
+&\left.\left\{\begin{array}{ll}c_{n,m}B_{n,\frac m2}\Delta^\frac m2H(\xi)&\hbox{if }n\hbox{ and }m\hbox{ even}\\0&\hbox{if }n\hbox{ and }m\hbox{ odd}\\(-1)^\frac{m-n+1}2D_{n,m}\mathsf J_{n,m}(H)&\hbox{otherwise}\end{array}\right.\)
\end{align*}

\end{definition}\

The symbol $a\lesssim b$ will be used to mean $a\le cb$ with $c$ independent on the quantities.\\
We denote as $\partial_\nu$ the (outer) normal derivative of a function at a point on $\dBn$ and as $\Delta_\tau$ the tangential Laplacian.\\
For a multi-index $\a=(\a_1,\dots,\a_n)\in\N^n$ we denote:
$$|\a|:=\a_1+\dots+\a_n;\qquad\qquad x^\a:=x_1^{\a_1}\cdot\dots\cdot x_n^{\a_n};\qquad\qquad\partial_{x_\a}:=\de_{x_1}^{\a_1}\dots\de_{x_n}^{\a_n}.$$\

\subsection{Conformal Metrics}\

Throughout this article we will use the existing conforming equivalence between $\R^n_+$ and $\Bn$ via the inversion map \eqref{inv}, often without explicitly specifying it. Therefore, it is important to remember the conformal properties of the conformal Laplacian and conformal boundary operator. 

\smallskip

If $n\ge3$ and $\tilde g=\rho^\frac4{n-2}g$ is a conformal metric, then the conformal Laplacian and conformal boundary operators, defined by
\begin{equation*}
L_g=-\frac{4(n-1)}{n-2}\Delta_g+k_g,\quad
B_g=\frac2{n-2}\de_\nu+h_g,
\end{equation*}
are conformally invariant in the following sense: 
\begin{equation}\label{conformalinv}
L_g\varphi=\rho^\frac{n+2}{n-2}L_{\tilde g}\(\frac\varphi\rho\),\quad B_g\varphi=\rho^\frac n{n-2}B_{\tilde g}\(\frac\varphi\rho\).
\end{equation}
If $n=2$, then the Laplace-Beltrami operator and the normal derivative satisfy the following conformal property: if $\tilde g=e^\rho g$ is a conformal metric, then
\begin{equation*}
\Delta_{e^\rho g}=e^{-\rho}\Delta_g,\quad\nabla_{e^\rho g}\cdot\eta_{e^\rho g}=e^{-\frac\rho2}\nabla\cdot\eta_g.
\end{equation*}
The following result establishes the conformal invariance of a certain geometric quantity that will be very much related to our energy functionals.
\begin{lemma}\label{confinv} Let $(M^n,g)$ be a compact Riemannian manifold of dimension $n\ge3$ and $\tilde g=\vfi^{\frac{4}{n-2}}g$ a conformal metric with $\vfi$ smooth and positive. If we set $\hat f=f\vfi^{-1}$, then
\begin{align}\label{eqconf}
&\frac{4(n-1)}{n-2}\int_M\(\nabla_{\tilde g}\hat u\cdot\nabla_{\tilde g}\hat v\)dV_{\tilde g}+\int_M k_{\tilde g}\hat u\hat v\:dV_{\tilde g}+2(n-1)\int_{\de M}h_{\tilde g}\hat u\hat v\:d\sigma_{\tilde g}\\
\nonumber=&\frac{4(n-1)}{n-2}\int_M\(\nabla_gu\cdot\nabla_gv\)dV_g+\int_Mk_guv\:dV_g+2(n-1)\int_{\de M}h_g uv\:d\sigma_g.
\end{align}
\end{lemma}
\begin{proof}
We will use the following basic identities:
$$dV_{\tilde g}=\vfi^\dst dV_g, d\sigma_{\tilde g}=\vfi^\dsh d\sigma_g,\nabla_{\tilde g}=\vfi^{-\frac{4}{n-2}}\nabla_g,$$
where $\dst,\dsh$ are as in \eqref{exp} and the relation between $k_{\tilde g}$, $k_g$, $h_{\tilde g}$ and $h_g$ given by \eqref{sm}. The first term in the left-hand side of \eqref{eqconf} can be decomposed using the previous identities:
\begin{align}\label{eqconf1}
&\int_M\(\nabla_{\tilde g}\hat u\cdot\nabla_{\tilde g}\hat v\)dV_{\tilde g}=\int_M\vfi^2\(\nabla_g\hat u\cdot\nabla_g\hat v\)dV_g\nonumber\\
=&\int_M\(\nabla_gu\cdot\nabla_gv\)dV_g-\int_M\(\hat v\nabla_g\vfi\cdot\nabla_gu+\hat u\nabla_g\vfi\cdot\nabla_gv-\hat u\hat v\abs{\nabla_g\vfi}^2\)dV_g.
\end{align}
On the other hand, integrating by parts on $M$ and using \eqref{sm}:
\begin{align}\label{eqconf2}
&\int_Mk_{\tilde g}\hat u\hat v\:dV_{\tilde g}=\int_M\hat u\hat v\(k_g\vfi^2-\frac{4(n-1)}{n-2}(\Delta_g\vfi)\vfi\)dV_g\nonumber\\
=&\int_Mk_guv\:dV_g-2(n-1)\int_{\de M}h_{\tilde g}\hat u\hat v\vfi^\dsh d\sigma_g+2(n-1)\int_{\de M}h_g uv\:d\sigma_g\nonumber\\
+&\frac{4(n-1)}{n-2}\int_M\(\hat v\nabla_g\vfi\cdot\nabla_gu+\hat u\nabla_g\vfi\cdot\nabla_gv-\hat u\hat v\abs{\nabla_g\vfi}^2\)dV_g.
\end{align}
Finally, \eqref{eqconf} can be obtained combining \eqref{eqconf1} and \eqref{eqconf2}.
\end{proof}

\subsection{Solutions of the unperturbed problems}\

By means of the inversion map and the classification results available for $\R^n_+$, we can give an $n-$dimensional family of solutions for the problems $\eqref{pe}$ and $\eqref{pne}$ with $\e=0$. 

\smallskip

First, we consider the problem in $\Bd$:
\begin{equation}\label{unpd}\tag{$P^2_0$}
\left\{\begin{array}{ll}
-\Delta u=-2e^u &\hbox{in }\quad\Bd\\
\de_\nu u+2=2\D e^{\frac u2}&\hbox{on }\quad\Sph^1.
\end{array}\right.
\end{equation}
By \cite{zhang}, a family of solutions of the problem in the half space
\begin{equation}\label{unpr2}
\left\{\begin{array}{ll}
-\Delta u=-2e^u &\hbox{in }\quad\R^2_+\\
\de_\nu u=2\D e^{\frac u2}&\hbox{on }\quad\de\R^2_+.
\end{array}\right.
\end{equation}
is given by 
\begin{equation*}
U_{x_0,\l}(x,y)=2\log\frac{2\l}{(x-x_0)^2+(y+\D\l)^2-\l^2},
\end{equation*}
for $\l>0$ and $x_0\in\R$. Other classification results for solutions to \eqref{unpr2} are given in \cite{galvez-mira,li-zhu95}.\\
Let us call $\hat U_{x_0,\l}=\(U_{x_0,\l}-\rho\)\circ\inv^{-1}$. Taking into account equation \eqref{rho2} and the conformal properties of the Laplacian and normal derivative in $\R^2$, it is clear that
\begin{equation*}
\left\{\begin{array}{ll}
-\Delta\hat U=-2e^{\hat U}&\hbox{in }\quad\Bd\\
\de_\nu\hat U+2=2\D e^\frac{\hat U}2&\hbox{on }\quad\Sph^1.
\end{array}\right.
\end{equation*}
Therefore, a family of solutions for \eqref{unpd} is given by
\begin{equation}\label{bubblesdisk}
\hat U_{x_0,\l}(s,t)=2\log\frac{\l\(x^2+(y+1)^2\)}{(x-x_0)^2+(y+\D\l)^2-\l^2},
\end{equation}
with $$x=x(s,t)=\frac{2s}{s^2+(t+1)^2},\quad y=y(s,t)=\frac{1-s^2-t^2}{s^2+(t+1)^2},$$ $\l>0$ and $x_0\in\R.$\\ 

Now, we address the unperturbed problem in $\Bn$ for $n\ge3$:
\begin{equation}\label{problem2}\tag{$P^{n\,'}_0$}
\left\{\begin{array}{ll}-\frac{4(n-1)}{n-2}\Delta u=-u^\frac{n+2}{n-2}&\hbox{in }\Bn,\\\frac2{n-2}\de_\nu u+u=\frac{\D}{\sqrt{n(n-1)}}u^\frac n{n-2}&\hbox{on }\dBn.
\end{array}\right.
\end{equation}
Consider $\R^n_+$ with its usual metric, and the problem
\begin{equation}\label{problemrn}
\left\{\begin{array}{ll}\frac{-4(n-1)}{n-2}\Delta u=-u^\frac{n+2}{n-2}&\hbox{in }\R^n_+,\\
-\frac2{n-2}\de_{x_n}u=\frac{\D}{\sqrt{n(n-1)}}u^\frac n{n-2}&\hbox{on }\de\R^n_+.
\end{array}\right.
\end{equation}
The results in \cite{cfs96} imply that all solutions of \eqref{problemrn} have the form
\begin{equation*}\label{bubble}
U_{x_0,\l}(\bar x,x_n)=\frac{(4n(n-1))^\frac{n-2}4\l^\frac{n-2}2}{(\abs{\bar x-x_0}^2+(x_n+\l\D)^2-\l^2)^\frac{n-2}2},
\end{equation*}
for any $x_0\in\de\R^n_+$ and $\l>0$.\\

Then, by \eqref{conformalinv}, we can write \eqref{problemrn} as:
\begin{equation*}
\left\{\begin{array}{ll}\rho^\frac{n+2}{n-2}\frac{-4(n-1)}{n-2}\Delta\(\frac u\rho\)=-u^\frac{n+2}{n-2}&\hbox{in }\R^n_+,\\\rho^\frac n{n-2}\(\frac2{n-2}\de_\nu\(\frac u\rho\)+\(\frac u\rho\)\)=\frac{\D}{\sqrt{n(n-1)}}u^\frac n{n-2}&\hbox{on }\de\R^n_+.
\end{array}\right.
\end{equation*}
If we call $\hat u=\(\frac u\rho\)\circ\inv^{-1}$, it is clear that
$$\left\{\begin{array}{ll}\frac{-4(n-1)}{n-2}\Delta\hat u=-{\hat u}^\frac{n+2}{n-2}&\hbox{in }\Bn,\\\frac2{n-2}\de_\nu\hat u+\hat u=\frac{\D}{\sqrt{n(n-1)}}{\hat u}^\frac n{n-2}&\hbox{on }\dBn,
\end{array}\right.$$
which is exactly \eqref{problem2}. Hence, a family of solutions of \eqref{problem2} is given by
\begin{equation}\label{bubblespalla}
\hat U_{x_0,\l}(\bar x,x_n)=\l^\frac{n-2}2(n(n-1))^\frac{n-2}4\(\frac{\abs{\bar z}^2+(z_n+1)^2}{\abs{\bar z-x_0}^2+(z_n+\l\D)^2-\l^2}\)^\frac{n-2}2,
\end{equation}
with 
\begin{equation*}
\bar z=\bar z(\bar x,x_n)=\frac{2\bar x}{\abs{\bar x}^2+(x_n+1)^2},\quad z_n=z_n(\bar x,x_n)=\frac{1-\abs{\bar x}^2-x_n^2}{\abs{\bar x}^2+(x_n+1)^2}.
\end{equation*}

In view of formulae \eqref{bubblesdisk} and \eqref{bubblespalla} we set:
\begin{equation*}
P_{x_0,\l}(\bar x, x_n)=\frac{\L_n\l^2\(\abs{\bar z}^2+(z_n+1)^2\)^2}{\(\abs{\bar z-x_0}^2+(z_n+\l\D)^2-\l^2\)^2},
\end{equation*}
with $\bar z,z_n$ as before, $x_0\in\R^{n-1}$, $\l>0$, and define
\begin{equation}\label{bubbles}
V_{x_0,\l}=\left\{\begin{array}{ll}
{P_{x_0,\l}}^\frac{n-2}4&\hbox{if}\quad n\ge3,\\
\log P_{x_0,\l}&\hbox{if}\quad n=2.
\end{array}\right.
\end{equation}

\

\section{Properties of the Energy Functionals}\label{sec-energy}
We define the functionals $J^n_\e:H^1\(\Bn\)\to\R$ by
\begin{align}\label{funct-2}
J^2_\e(u)=&\frac12\int_\Bd\abs{\nabla u}^2+2\int_{\Sph^1}u+2\int_\Bd(1+\e K)e^u-4\D\int_{\Sph^1}(1+\e H)e^{\frac u2},\\
\nonumber J^n_\e(u)=&\frac12\int_\Bn\abs{\nabla u}^2+\frac12\int_{\dBn}u^2+\frac{(n-2)^2}{8n(n-1)}\int_\Bn(1+\e K)\abs{u}^\dst\\
-&\frac{(n-2)^2}{4\sqrt{n(n-1)^3}}\D\int_{\dBn}(1+\e H)\abs{u}^\dsh,\quad\hbox{if}\quad n\ge3.\label{funct-n}
\end{align}
Observe that we can write 
\begin{equation*}
J^n_\e(u)=J^n_0(u)+\e\a_n\gamma^n(u),
\end{equation*}
with
\begin{equation*}
\gamma^n(u)=\left\{\begin{array}{ll}\int_\Bd Ke^u-\b_2\D\int_{\Sph^1}He^\frac u2&\hbox{if }n=2\\\int_\Bn K\abs{u}^\dst-\b_n\D\int_{\dBn}H\abs{u}^\dsh&\hbox{if }n\ge3\end{array}\right.
\end{equation*}
with $\a_n,\b_n$ as in Definition \ref{def-constants}.\\
Let $V_{x_0,\l}$ be given by \eqref{bubbles}. We set
$$\Gamma(x_0,\l)=\gamma^n\(V_{x_0,\l}\)=\int_\Bn K{P_{x_0,\l}}^\frac n2-\b_n\D\int_\dBn H{P_{x_0,\l}}^\frac{n-1}2.$$
The first term of the energy is constant along our family of solutions:
\begin{proposition}
There exist constants $\mathtt E_{n,\D}$, independent on $\l$ and $x_0$, such that
\begin{equation*}
J_0^n(V_{x_0,\l})=\mathtt E_{n,\D},\quad\forall n\ge2.
\end{equation*}
\end{proposition}
\begin{proof}
Let us study the cases $n=2$ and $n\ge3$ separately.\\
When $n=2$, integrating by parts and using \eqref{unpd} and \eqref{rho2}, we can see that:
\begin{align*}
&\frac12\int_\Bd\abs{\nabla V_{x_0,\l}}^2+2\int_{\Sph^1}V_{x_0,\l}=\frac12\int_{\R^2_+}\abs{\nabla\(U_{x_0,\l}-\rho\)}^2+2\int_\R(U_{x_0,\l}-\rho)e^\frac\rho2\\
=&-\frac12\int_{\R^2_+}\Delta U_{x_0,\l}(U_{x_0,\l}-\rho)+\frac12\int_\R\de_\nu U_{x_0,\l}(U_{x_0,\l}-\rho)+2\int_\R(U_{x_0,\l}-\rho)e^\frac\rho2\\
=&\frac12\int_{\R^2_+}\abs{\nabla U_{x_0,\l}}^2+\frac12\int_{\R^2_+}\Delta U_{x_0,\l}\,\rho-\frac12\int_\R\de_\nu U_{x_0,\l}\rho+2\int_\R(U_{x_0,\l}-\rho)e^\frac\rho2\\
=&\frac12\int_{\R^2_+}\abs{\nabla U_{0,1}}^2-\frac12\int_{\R^2_+}\abs{\nabla\rho}^2. 
\end{align*}
Now,
\begin{align*}
2\int_\Bd e^{V_{x_0,\l}}-4\D\int_{\Sph^1}e^\frac{V_{x_0,\l}}2=2\int_{\R^2_+}e^{U_{0,1}}-4\D\int_\R e^\frac{U_{0,1}}2.
\end{align*}
Finally,
\begin{equation}\label{constant-2}
J^2_0(V_{x_0,\l})=\frac12\int_{\R^2_+}\(\abs{\nabla U_{0,1}}^2-\abs{\nabla\rho}^2\)+2\int_{\R^2_+}e^{U_{0,1}}-4\D\int_\R e^\frac{U_{0,1}}2.
\end{equation}
As for the case $n\ge3$, from Lemma \ref{confinv} it follows
\begin{equation*}
\frac12\int_\Bn\abs{\nabla V_{x_0,\l}}^2+\frac12\int_{\dBn}{V_{x_0,\l}}^2=\frac12\int_{\R^n_+}\abs{\nabla U_{x_0,\l}}^2=\frac12\int_{\R^n_+}\abs{\nabla U_{0,1}}^2.
\end{equation*}
Moreover, by a direct change of variables, we obtain
\begin{align*}
\int_\Bn{V_{x_0,\l}}^\dst-\b_n\D\int_{\dBn}{V_{x_0,\l}}^\dsh=\int_{\R^n_+}{U_{0,1}}^\dst-\b_n\D\int_{\de\R^n_+}{U_{0,1}}^\dsh.
\end{align*}
Therefore,
\begin{equation}\label{constant-n}
J^n_0(V_{x_0,\l})=\frac12\int_{\R^n_+}\abs{\nabla U_{0,1}}^2+\a_n\(\int_{\R^n_+}{U_{0,1}}^\dst-\b_n\D\int_{\de\R^n_+}{U_{0,1}}^\dsh\).
\end{equation}
\end{proof}

By a change of variables and using the relations in Section \ref{sec-preli}, we can move to $\R^n_+$ and write our function $\Gamma$ in a more suitable way.
\begin{proposition}It holds
\begin{align}\label{gamma-def}
\Gamma(x_0,\l)=&\int_{\R^n_+}\frac{\L_n^\frac n2\tilde K(\bar x,x_n)\l^nd\bar xdx_n}{\(\abs{\bar x-x_0}^2+(x_n+\l\D)^2-\l^2\)^n}-\int_{\de\R^n_+}\frac{\L_n^\frac{n-1}2\b_n\D\tilde H(\bar x)\l^{n-1}d\bar x}{\(\abs{\bar x-x_0}^2+\l^2(\D^2-1)\)^{n-1}}\\
\nonumber=&\int_{\R^n_+}\frac{\L_n^\frac n2\tilde K(\l\bar y+x_0,\l y_n)}{\(\abs{\bar y}^2+(y_n+\D)^2-1\)^n}d\bar ydy_n-\int_{\de\R^n_+}\frac{\L_n^\frac{n-1}2\b_n\D\tilde H(\l\bar y+x_0)}{\(\abs{\bar y}^2+\D^2-1\)^{n-1}}d\bar y,
\end{align}
where $\tilde K=K\circ\inv,\tilde H=H\circ\inv$.
\end{proposition}
We are interested in the behaviour of $\Gamma$ at infinity and when $\l\to0$.
\begin{proposition}\label{gammainf}
$\lim_{\abs{x_0}+\l\to+\infty}\Gamma(x_0,\l)=\psi((0,-1))$
\end{proposition}
\begin{proof}
First, notice that 
\begin{align*}
\lim_{\l+\abs{x_0}\to+\infty}\inv(\l\bar x+x_0,\l x_n)=&\lim_{\l+\abs{x_0}\to+\infty}\(\frac{2(\l\bar x+x_0)}{\abs{\l\bar x+x_0}^2+(\l x_n+1)^2},\frac{1-|\l\bar x+x_0|^2-(\l x_n)^2}{\abs{\l\bar x+x_0}^2+(\l x_n+1)^2}\)\\=&\left\{\begin{array}{ll}
(0,-1)&\hbox{locally uniformly on}\quad(\bar x,x_n)\neq(0,0),\\
\inv(x_0,0)&\hbox{if}\quad(\bar x,x_n)=(0,0).
\end{array}\right.
\end{align*}
With that in mind, we fix $\e>0$ small enough and write
\begin{align*}
\Gamma(x_0,\l)=&\int_{|y|>\epsilon}\frac{\L_n^\frac n2\tilde K(\l\bar y+x_0,\l y_n)}{\(\abs{\bar y}^2+(y_n+\D)^2-1\)^n}d\bar ydy_n-\int_{|\bar y|>\epsilon}\frac{\L_n^\frac{n-1}2\b_n\D\tilde H(\l\bar y+x_0,0)}{\(\abs{\bar y}^2+\D^2-1\)^{n-1}}d\bar y\\
+&\int_{|y|\le\epsilon}\frac{\L_n^\frac n2\tilde K(\l\bar y+x_0,\l y_n)}{\(\abs{\bar y}^2+(y_n+\D)^2-1\)^n}d\bar ydy_n-\int_{|\bar y|\le\epsilon}\frac{\L_n^\frac{n-1}2\b_n\D\tilde H(\l\bar y+x_0,0)}{\(\abs{\bar y}^2+\D^2-1\)^{n-1}}d\bar y.
\end{align*}
Then, taking limits when $\l+\abs{x_0}\to+\infty$,
\begin{align*}
\Gamma(x_0,\l)=& K(0,-1)\int_{|y|>\epsilon}\frac{\L_n^\frac n2d\bar ydy_n}{\(\abs{\bar y}^2+(y_n+\D)^2-1\)^n}\\
-&H(0,-1)\int_{|\bar y|\le\epsilon}\frac{\L_n^\frac{n-1}2\b_n\D d\bar y}{\(\abs{\bar y}^2+\D^2-1\)^{n-1}}+O\(\epsilon^{n-1}\)\\
+& K(\inv(x_0,0))O\(\epsilon^n\)-H(\inv(x_0,0))O\(\epsilon^{n-1}\).
\end{align*}
The claim follows from taking limits when $\epsilon\to0$.
\end{proof}
The following result describes the behaviour of $\Gamma$ around $\l=0$. Its proof will postponed to Appendix \ref{section-app-A}.
\begin{proposition}\label{coro-deriv}Define $\psi:\dBn\to\R$ by $\psi(\xi):=\mathtt a(\D_n)K(\xi)-\mathtt b(\D_n)H(\xi)$, and let us write $\xi=\inv(x_0)\in\dBn$. The following expansions hold, for any $m\in\N$, when $\l\ll1$:\\
If $n=2$,
\begin{align*}
\Gamma(x_0,\l)=&\psi(\xi)-\(2\pi(1+\xi_n)\l\(\(\D-\sqrt{\D^2-1}\)\,\de_\nu K(\xi)-2\D(-\Delta)^\frac12H(\xi)\)\right.\\
+&\left.\l^2\log\frac1\l\Phi_3(\xi)+\l^2\Phi_4(\xi)\dots+\l^m\log\frac1\l\Phi_{2m-1}(\xi)+\l^m\Phi_{2m}(\xi)\)(1+o(1));
\end{align*}
If $n\ge3$,
\begin{align*}
\Gamma(x_0,\l)=&\psi(\xi)-\(a_{n,0,1}(1+\xi_n)\l\de_\nu K(\xi)+\l^2\Phi_2(\xi)+\dots+\l^{n-2}\Phi_{n-2}(\xi)\right.\\
+&\left.\l^{n-1}\log\frac1\l\Phi_{n-1}(\xi)+\l^{n-1}\Phi_n(\xi)\dots+\l^m\log\frac1\l\Phi_{2m-n+1}(\xi)+\l^m\Phi_{2m-n+2}(\xi)\)(1+o(1)).
\end{align*}

\

Here $\mathtt a(\D_n),\mathtt b(\D_n),a_{n,0,1},\Phi_j(\xi)$ are given in Definition \ref{def-constants}.
\end{proposition}

\

\section{The Linear Theory}\label{sec-linear}
In this section we develop the technicalities of the Ljapunov-Schmidt finite dimensional reduction. Most of the results hereby presented are well-known in the literature of this argument, therefore details of the proofs will be skipped.

\subsection{The $2-$dimensional case}\

It is known (see \cite{jlsmr}) that the solutions of the linear problem
\begin{equation*}\label{lin2}
\left\{\begin{array}{ll}
-\Delta\psi+2e^{U_{x_0,\l}}\psi=0&\hbox{in }\Bd\\
\de_\nu\psi-\D e^\frac{U_{x_0,\l}}2\psi=0&\hbox{on }\Sph^1\\
\end{array}\right.\end{equation*}
are a linear combination of 
\begin{equation*}\label{zeta2}
\mathcal Z^1_{x_0,\l}(z):=\de_{x_0}U_{x_0,\l}\quad\hbox{and}\quad\mathcal Z^2_{x_0,\l}(z):=\de_\l U_{x_0,\l}\quad\hbox{and}\quad\end{equation*}

Given $\kappa>0$, set
\begin{equation}\label{ck2}\mathtt C_\kappa:=\left\{(t,x_0,\l)\in\mathbb R\times(0,\infty)\times\mathbb R^{n-1}\ :\ \frac1\kappa\le|t|+\l\le\kappa,\ |x_0|\le\kappa\right\}.\end{equation}

Arguing as in Theorem 3.3 of \cite{bmp} we can prove that

\begin{proposition}\label{key2}
Fix $p>1$ and $\kappa>0$. For any $(x_0,\l)\in\mathtt C_\kappa$ (see \eqref{ck2}) and $\mathfrak f\in L^p\(\Bd\)$ and $\mathfrak g\in L^p\(\Sph^1\)$ such that
\begin{equation*}\label{cd2}
\int_\Bd\mathfrak f+\int_{\Sph^1}\mathfrak g=\int_\Bd\mathfrak f\mathcal Z_{x_0,\l}^i+\int_{\Sph^1}\mathfrak g\mathcal Z_{x_0,\l}^i=0,\quad i=1,2,
\end{equation*}
there exists a unique $\phi\in H^1\(\Bd\)$ such that
\begin{equation}\label{orto2}-2\int_\Bd e^{U_{x_0,\l}}\phi+\D\int_{\Sph^1}e^\frac{U_{x_0,\l}}2\phi=-2\int_\Bd e^{U_{x_0,\l}}\phi\mathcal Z_{x_0,\l}^i+\D\int_{\Sph^1}e^\frac{U_{x_0,\l}}2\phi\mathcal Z_{x_0,\l}^i=0,\ i=1,2,\end{equation}
which solves the problem
$$\begin{cases}-\Delta\phi+2e^{U_{x_0,\l}}\phi=\mathfrak f&\hbox{in }\Bd\\
\de_\nu\phi-\D e^\frac{U_{x_0,\l}}2\phi=\mathfrak g&\hbox{on }\Sph^1\\
\end{cases}$$
Furthermore
\begin{equation*}\label{k2}
\|\phi\|\lesssim\(\|\mathfrak f\|_{L^p\(\Bd\)}+\|\mathfrak g\|_{L^p\(\Sph^1\)}\).
\end{equation*}
\end{proposition}

\subsubsection{Rewriting the problem}\

We look for a solution of \eqref{pe} in the form
$$u=U_{ x_0,\l}+\tau+\phi,\ \hbox{with}\ \l>0,\ x_0\in\mathbb R\quad\hbox{and}\quad\tau=t\sqrt\e,\ t\in\mathbb R$$
where $\phi$ satisfies the orthogonality condition \eqref{orto2}. We shall rewrite problem \eqref{pe} as a system
\begin{equation}
\label{s2}\left\{\begin{array}{ll}-\Delta\phi+2e^{U_{x_0,\l}}\phi=\mathscr E_{in}+\mathscr N_{in}(\phi)+c_0+\sum\limits_{i=1,2}c_i\mathcal Z_{x_0,\l}^i &\hbox{in }\Bd\\
\de_\nu\phi-\D e^\frac{U_{x_0,\l}}2\phi=\mathscr E_{bd}+\mathscr N_{bd}(\phi)+c_0+\sum\limits_{i=1,2}c_i\mathcal Z_{x_0,\l}^i &\hbox{on }\Sph^1\end{array}\right.
\end{equation}
where $c_i$'s are real numbers.\\
The error that we are paying by using this approximating solution equals to
\begin{equation*}\label{errore}
\mathscr E_{in}:=-\e\mathcal F\(U_{x_0,\l}+\tau\)\quad\hbox{and}\quad\mathscr E_{bd}:=\e\mathcal G\(U_{x_0,\l}+\tau\)\end{equation*}
and the non-linear part is
\begin{align}
\mathscr N_{in}(\phi):=&-\left[\mathcal F\(U_{x_0,\l}+\tau+\phi\)-
\mathcal F\(U_{x_0,\l}+\tau\)-\mathcal F'\(U_{x_0,\l}\)\phi\right]-
\left[\(\mathcal F'\(U_{x_0,\l}+\tau\)-\mathcal F'\(U_{x_0,\l}\)\)\phi\right]\\
\nonumber&-\e K\left[\mathcal F\(U_{x_0,\l}+\tau+\phi\)-
\mathcal F\(U_{x_0,\l}+\tau\)\right];\\ 
\nonumber\mathscr N_{bd}(\phi):=&-\left[\mathcal G\(U_{x_0,\l}+\tau+\phi\)-
\mathcal G\(U_{x_0,\l}\)-\mathcal G'\(U_{x_0,\l}\)\phi\right]-\left[\(\mathcal G'\(U_{x_0,\l}+\tau\)-
\mathcal G'\(U_{x_0,\l}\)\)\phi\right]\\
\nonumber&-\e H\left[\mathcal G\(U_{x_0,\l}+\phi\)-
\mathcal G\(U_{x_0,\l}\)\right]).\end{align}
Here we set
$$\mathcal F(u):=2e^u\quad\hbox{and}\quad\mathcal G(u)=2\D e^\frac u2.$$

We have the following result
\begin{proposition}\label{rido2}
Fix $\kappa>0$. There exists $\e_\kappa>0$ such that or any $(x_0,\l)\in\mathtt C_\kappa$ (see \eqref{ck2}) there exists a unique $\phi=\phi(\e,x_0,\l)\in H^1\(\Bd\)$ and $c_i\in\mathbb R$ which solve \eqref{s2}.
Moreover, $(x_0,\l)\to\phi(\e,x_0,\l)$ is a $C^1-$function and $\|\phi\|\lesssim\e.$
\end{proposition}
\begin{proof}The proof is standard and relies on a contraction mapping argument combined with the linear theory developed in Proposition \ref{key2} and the estimates for $p>1$
$$\|\mathscr E_{in}\|_{L^p\(\Bd\)}\lesssim\e\quad\hbox{and}\quad\|\mathscr E_{bd}\|_{L^p\(\Sph^1\)}\lesssim\e.$$
\end{proof}

\subsubsection{The reduced energy}\

Let us consider the energy functional $J_\e^2$ defined in \eqref{funct-2}, whose critical points produce solutions of \eqref{pe}. We define the reduced energy
$$\widetilde J^2_\e(t,x_0,\l):=J^2_\e\(U_{x_0,\l}+t+\phi\),$$
where $\phi$ is given in Proposition \ref{rido2}.
\begin{proposition}\label{ene2}The following are true:
\begin{enumerate}
\item If $(x_0,\l)$ is a critical point of $\widetilde J_\e$, then $U_{x_0,\l}+\phi$ is a solution to \eqref{pe}.
\item The following expansion holds
$$\widetilde J_\e(t,x_0,\l)=\mathtt E_{2,\D}-\e\(\a_\D t^2+\Gamma(x_0,\l)\)+o(\e)$$
$C^1-$uniformly in compact sets of $\mathbb R\times(0,+\infty)\times\mathbb R$.\\
Here $\mathtt E_{2,\D}$ is a constant independent on $x_0$, $t$ and $\l$ whose expression is given by \eqref{constant-2}, $\Gamma$ is defined in \eqref{gamma-def} and
$$\a_\D=\pi\(\frac\D{\sqrt{\D^2-1}}-2\).$$
\end{enumerate}
\end{proposition}
\begin{proof}
We use the choice $\tau=t\sqrt\e$ and the fact that 
$$\D\int_{\Sph^1}e^\frac{U_{x_0,\l}}2d\sigma-\int_\Bd e^{U_{x_0,\l}}dx=2\pi\quad\hbox{and}\quad
\int_{\Sph^1}e^\frac{U_{x_0,\l}}2d\sigma=2\pi.$$
\end{proof}

\subsection{The $n-$dimensional case}\

Recently, in \cite{cpv}, it has been proved that all the solutions to the linearized problem
\begin{equation*}\label{lin-n}
\left\{\begin{array}{ll}
-\frac{4(n-1)}{n-2}\Delta Z=-\frac{n+2}{n-2}{U_{x_0,\l}}^\frac4{n-2}Z&\hbox{in }\mathbb B^n,\\
\frac2{n-2}\de_\nu Z+Z=\frac n{(n-2)\sqrt{n(n-1)}}\D{ U_{x_0,\l}}^\frac2{n-2}Z&\hbox{on }\dBn\end{array}\right.
\end{equation*}
are a linear combination of the $n$ functions
\begin{equation*}\label{zeta-n}
Z^i_{x_0,\l}=\de_{x_{0,i}}U_{x_0,\l},\ i=1,\dots,n-1\quad\hbox{and}\quad Z^n_{x_0,\l}=\de_\l U_{x_0,\l}.
\end{equation*}

Given $\kappa>0$ set
\begin{equation}\label{ckn}\mathtt C_\kappa:=\left\{( x_0,\l)\in(0,\infty)\times\mathbb R^{n-1}\ :\ \frac1\kappa\le\l\le\kappa,\ |x_0|\le\kappa\right\}.\end{equation}

Arguing as in \cite{cpv} we can prove that
\begin{proposition}\label{keyn}
Fix $\kappa>0$. For any $(x_0,\l)\in\mathtt C_\kappa$ (see \eqref{ckn}) and $\mathfrak f\in L^\frac{2n}{n+2}\(\mathbb B^n\)$ and $\mathfrak g\in L^\frac{2(n-1)}n\(\dBn\)$ such that
\begin{equation*}\label{cdn}
\int_{\mathbb B^n}\mathfrak f\mathcal Z_{x_0,\l}^i+\int_{\dBn}\mathfrak g\mathcal Z_{x_0,\l}^i=0,\quad i=1,\dots,n,
\end{equation*}
there exists a unique $\phi\in H^1\(\mathbb B^n\)$ such that
\begin{equation}\label{orton}-\frac{n+2}{n-2}\int_{\mathbb B^n}{U_{x_0,\l}}^\frac4{n-2}\phi\mathcal Z_{x_0,\l}^i+\frac n{(n-2)\sqrt{n(n-1)}}\D\int_{\dBn}{U_{x_0,\l}}^\frac2{n-2}\phi\mathcal Z_{x_0,\l}^i=0,\quad i=1,\dots,n,\end{equation}
which solves the problem
$$\begin{cases}-\Delta\phi+\frac{n+2}{4(n-1)}{U_{x_0,\l}}^\frac4{n-2}\phi=\mathfrak f&\hbox{in }\;\mathbb B^n\\
\de_\nu\phi+\frac{n-2}2\phi-\frac n{2\sqrt{n(n-1)}}\D{U_{x_0,\l}}^\frac2{n-2}\phi=\mathfrak g&\hbox{on }\;\dBn\\
\end{cases}$$
Furthermore
\begin{equation*}\label{kn}
\|\phi\|\lesssim\(\|\mathfrak f\|_{L^\frac{2n}{n+2}\(\mathbb B^n\)}+\|\mathfrak g\|_{L^\frac{2(n-1)}n\(\dBn\)}\).
\end{equation*}
\end{proposition}

\subsubsection{Rewriting the problem}\

We look for a positive solution of \eqref{pne} as
$$u=U_{ x_0,\l}+\phi\ \hbox{with}\ \l>0,\ x_0\in\mathbb R $$
where $\phi$ satisfies \eqref{orton}. We rewrite problem \eqref{pne} as a system
\begin{equation}\label{sn}\left\{\begin{array}{ll}-\Delta\phi+\frac{n+2}{4(n-1)}{U_{x_0,\l}}^\frac4{n-2}\phi=\mathscr E_{in}+\mathscr N_{in}(\phi)+\sum\limits_{i=1}^nc_i\mathcal Z_{x_0,\l}^i&\hbox{in }\\Bd\\
\de_\nu\phi+\frac{n-2}2\phi-\frac n{2\sqrt{n(n-1)}}\D{U_{x_0,\l}}^\frac2{n-2}\phi=\mathscr E_{bd}+\mathscr N_{bd}(\phi)+\sum\limits_{i=1}^nc_i\mathcal Z_{x_0,\l}^i&\hbox{in }\\Sph^1\\
\end{array}\right.\end{equation}
where the $c_i$ are real numbers. Moreover, the error is given by
\begin{equation*}\label{errore2}
\mathscr E_{in}:=-\e\mathcal F\(U_{x_0,\l}\)\quad\hbox{and}\quad\mathscr E_{bd}:=\e\mathcal G\(U_{x_0,\l}\)\end{equation*}
and the non-linear part is
\begin{align}
\mathscr N_{in}(\phi):=&-\left[\mathcal F\(U_{x_0,\l}+\phi\)-
\mathcal F\(U_{x_0,\l}\)-\mathcal F'\(U_{x_0,\l}\)\phi\right]-\e K\left[\mathcal F\(U_{x_0,\l}+\phi\)-
\mathcal F\(U_{x_0,\l}\)\right]\\ 
\nonumber\mathscr N_{bd}(\phi):=&-\left[\mathcal G\(U_{x_0,\l}+\phi\)-
\mathcal G\(U_{x_0,\l}\)-\mathcal G'\(U_{x_0,\l}\)\phi\right]-\e H\left[\mathcal G\(U_{x_0,\l}+\phi\)-
\mathcal G\(U_{x_0,\l}\)\right])\end{align}
Here we set
$$\mathcal F(u):=-\frac{n-2}{4(n-1)}(u^+)^{\frac{n+2}{n-2}}\quad\hbox{and}\quad\mathcal G(u)=\frac{n-2}{2\sqrt{n(n-1)}}\D(u^+)^{\frac n{n-2}}.$$

We have the following result:
\begin{proposition}\label{ridon}
Fix $\kappa>0$. There exists $\e_\kappa>0$ such that or any $(x_0,\l)\in\mathtt C_\kappa$ (see \eqref{ckn}) there exists a unique $\phi=\phi(\e,x_0,\l)\in H^1\(\Bd\)$ and $c_i\in\mathbb R$ which solve \eqref{sn}.
Moreover, $(x_0,\l)\to\phi(\e,x_0,\l)$ is a $C^1-$function and $\|\phi\|\lesssim\e.$
\end{proposition}
\begin{proof} The proof is standard and relies on a contraction mapping argument combined with the linear theory developed in Proposition \ref{keyn} and the estimates 
$$\|\mathscr E_{in}\|_{L^\frac{2n}{n+2}\(\Bd\)}\lesssim\e\quad\hbox{and}\quad\|\mathscr E_{bd}\|_{L^\frac{2(n-1)}n\(\Sph^1\)}\lesssim\e.$$
\end{proof}

\subsubsection{The reduced energy}\

We consider the functional $J^n_\e$ defined on \eqref{funct-n}. It is easy to see that its critical points are positive solutions to equation \eqref{pne}. Now, we introduce the reduced energy
$$\widetilde J^n_\e(x_0,\l):=J^n_\e\(U_{x_0,\l}+\phi\),$$
where $\phi$ is given in Proposition \ref{ridon}.
It is quite standard to prove the following result
\begin{proposition}
The following assertions hold true
\begin{enumerate}
\item If $(x_0,\l)$ is a critical point of $\widetilde J_\e$, then $U_{x_0,\l}+\phi$ is a solution to \eqref{pne}.
\item Moreover, we have the following expansion
$$\widetilde J^n_\e(x_0,\l)=\mathtt E_{n,\D}-\e\Gamma(x_0,\l)+o(\e)$$
$C^1-$uniformly with respect to $(x_0,\l)$ in compact sets of $(0,+\infty)\times\mathbb R^{n-1}.$\\
Here $\mathtt E_{n,\D}$ is a constant independent on $x_0$ and $\l$, given by \eqref{constant-n}, and $\Gamma$ is the function defined on \eqref{gamma-def}
\end{enumerate}
\end{proposition}

\

\section{Existence of critical points of $\Gamma$}\label{sec-critical}

In this section we are finally able to get critical points of the map $(x_0,\l)\mapsto\Gamma(x_0,\l)$, hence solutions to problems \eqref{pe}, \eqref{pne}.\\

We start with the following abstract result about critical points of maps defined on balls in dependence of the boundary behavior.\\
\begin{proposition}\label{punticrit}
Let $f:\Bn\to\R$ be a $C^1$ map satisfying, as $\xi$ goes to $\dBn$,
$$f(\xi)=f_0\(\frac\xi{|\xi|}\)+g_{\frac\xi{|\xi|}}(1-|\xi|)f_1\(\frac\xi{|\xi|}\)+o\(g_{\frac\xi{|\xi|}}(1-|\xi|)\),$$
for some $f_i:\dBn\to\R$ with $f_0$ of class $C^1$ and some increasing $g_{\frac\xi{|\xi|}}:(0,1)\to(0,+\infty)$ such that $g_{\frac\xi{|\xi|}}(t)\underset{t\to0}\to0$.\\
If one of the following holds true:
\begin{enumerate}
\item $f_1(\xi)>0$ at any global maximum $\xi$ of $f_0$;
\item $f_1(\xi)<0$ at any global minimum $\xi$ of $f_0$;
\item $f_1(\xi)\ne0$ at any critical point $\xi$ of $f_0$, $f_0$ is Morse and
$$\sum_{\{\xi:\nabla f_0(\xi)=0,\,f_1(\xi)>0\}}(-1)^{\ind_\xi\nabla f_0}\ne1;$$
\end{enumerate}
then, $f$ has at least a stable critical point.
\end{proposition}\

Theorems \ref{th2d} and \ref{thn} will follow without much difficulty from this proposition and Proposition \ref{coro-deriv}.\\
\begin{proof}[Proof of Theorems \ref{th2d}, \ref{thn}]
We only consider the case of Theorem \ref{th2d}, since the same arguments also work for Theorem \ref{thn}.\\
Thanks to Proposition \ref{ene2}, we get a solutions to the problem \eqref{pe} whenever $\frac\D{\sqrt{\D^2-1}}-2\ne0$, that is $\D\ne\frac2{\sqrt3}$, and $(x_0,\l)$ is a stable critical point of $\Gamma$. After composing with $\inv$, this is equivalent to getting a critical point of the map $f(\xi)=\Gamma\(\inv^{-1}(\xi)\)$, which is well-defined and smooth in the whole $\overline\Bn$ thanks to Proposition \ref{gammainf}.\\
In view of Proposition \ref{coro-deriv}, $f$ satisfies the assumptions of Proposition \ref{punticrit} with
$$f_0=\psi,\qquad g_\xi(t)=\left\{\begin{array}{ll}t&\hbox{if }\nabla\psi(\xi)\ne0\\t^{\lfloor\frac{m+1}2\rfloor}\log^\frac{1-(-1)^m}2\frac1\lambda&\mbox{if }\nabla\psi(\xi)=0\end{array}\right.,\qquad f_1=\left\{\begin{array}{ll}-2\pi\Phi_1&\hbox{if }\nabla\psi(\xi)\ne0\\-\Phi_m&\mbox{if }\nabla\psi(\xi)=0\end{array}\right.,$$
with $m=m(\xi)$ as in Theorem \ref{th2d} (if $m(\xi)$ is not well-defined, as for minima of $\psi$ in case (1) or maxima of $\psi$ in case (2), one can just set $g_\xi(t)=t,f_1=-2\pi\Phi_1$).\\
Here, we used that $\l=\frac{1-|\xi|}{1+\xi_n}+o(1-|\xi|)$ and that
\begin{align*}
a_{2,0,0}=&4\int_{\R^2_+}\frac1{\(\bar y^2+(y_2+\D)^2-1\)^2}d\bar ydy_2=\frac{2\pi}{\sqrt{\D^2-1}}\(\D-\sqrt{\D^2-1}\)\\
b_{2,0}=&4\frac\D{\sqrt{\D^2-1}}\int_{\de\R^2_+}\frac1{\bar y^2+1}d\bar y=\frac{4\pi\D}{\sqrt{\D^2-1}}
\end{align*}
hence the two definitions of $\psi$ given in Theorem \ref{th2d} and Proposition \ref{coro-deriv} actually coincide.\\
Since $-2\pi<0$, then the assumptions on $K,H$ in Theorem \ref{ene2} are equivalent to the ones in Proposition \ref{punticrit}, hence they ensure existence of solutions.
\end{proof}\

To prove Proposition \ref{punticrit}, we will compute the Leray-Schauder degree of the map $f$.
\begin{proof}[Proof of Proposition \ref{punticrit}]
First of all, $f$ can be extended up to $\dBn$ as $f_0$. Since $g$ vanishes at $0$, this extension is continuous.\\
Assume $(1)$ holds and take an absolute maximum point $\xi_0$ for $f$ on $\overline\Bn$. To get a critical point for $f$ on $\Bn$ we suffice to show that $\xi_0\not\in\dBn$.\\
If $\xi_0\in\dBn$, we would have $f_1(\xi_0)>0$, therefore, for $0<t\ll1$ we would have
$$f((1-t)\xi_0)=f_0(\xi_0)+g_{\xi_0}(t)f_1(\xi_0)+o(g_{\xi_0}(t))>f_0(\xi_0),$$
contradicting the fact that $\xi_0$ is a maximum point.\\
If $(2)$ holds, then the same argument shows that the minimum of $f$ on $\overline\Bn$ lies in the interior of $\Bn$, therefore it is a critical point of $f$.\\
Assume now that $(3)$ holds. We consider the \emph{double} of $\overline\Bn$, namely the manifold obtained by gluing two copies of $\Bn$ along the boundary: $\frac{\overline\Bn\times\{0,1\}}\sim$, where $(\xi,0)\sim(\xi,1)$ for $\xi\in\dBn$. This manifold is clearly diffeomorphic to $\mathbb S^n$, hence we will identify it as $\mathbb S^n$.\\
$f$ can be naturally extended to $\tilde f:\mathbb S^n\to\R$ as $\tilde f(\xi,i)=f(\xi)$ for $i=0,1$. The extension is continuous and, after a suitable rescalement of $g$ close to $0$, of class $C^1$ (with vanishing normal derivative on the equator). Such a rescalement does not affect the presence of critical points to $\tilde f$, $f$ and $f|_{\dBn}=f_0$, which we will now investigate.\\
We use the Euler-Poincaré formula to compute the Leray-Schauder degree of $\tilde f$, which is a Morse function by assumption:
\begin{align*}
1+(-1)^n=\chi\(\mathbb S^n\)=&\sum_{\{\xi\in\dBn:\nabla\tilde f(\xi)=0\}}(-1)^{\ind_\xi\nabla\tilde f}+\sum_{\{\xi\not\in\dBn:\nabla\tilde f(\xi)=0\}}(-1)^{\ind_\xi\nabla\tilde f}\\
=&\sum_{\{\xi\in\dBn:\nabla\tilde f(\xi)=0\}}(-1)^{\ind_\xi\nabla\tilde f}+2\sum_{\{\xi\in\Bn:\nabla f(\xi)=0\}}(-1)^{\ind_\xi\nabla f}.
\end{align*}
To deal with the critical points on $\dBn$, we notice that they are exactly the same critical points of $f_0$, but their index may change, since each can be either a minimum or a maximum in the orthogonal direction; precisely:
\begin{align*}
f_1(\xi)>0\Rightarrow&\ind_\xi\nabla\tilde f=\ind_\xi\nabla f_0;\\
f_1(\xi)<0\Rightarrow&\ind_\xi\nabla\tilde f=\ind_\xi\nabla f_0+1.
\end{align*}
and so
$$\sum_{\{\xi\in\dBn:\nabla\tilde f(\xi)=0\}}(-1)^{\ind_\xi\nabla\tilde f}=\sum_{\{\xi:\nabla f_0(\xi)=0,\,f_1(\xi)>0\}}(-1)^{\ind_\xi\nabla f_0}-\sum_{\{\xi:\nabla f_0(\xi)=0,\,f_1(\xi)<0\}}(-1)^{\ind_\xi\nabla f_0}.$$
Therefore, applying again the Euler-Poincaré formula, this time to $f_0$ on $\dBn$, we get:
\begin{align*}
1-(-1)^n=&\chi\(\dBn\)\\
=&\sum_{\{\xi\in\dBn:\nabla f_0(\xi)=0\}}(-1)^{\ind_\xi\nabla f_0}\\
=&\sum_{\{\xi:\nabla f_0(\xi)=0,\,f_1(\xi)>0\}}(-1)^{\ind_\xi\nabla f_0}+\sum_{\{\xi:\nabla f_0(\xi)=0,\,f_1(\xi)<0\}}(-1)^{\ind_\xi\nabla f_0}.
\end{align*}

By summing the previous equalities we get:
$$\sum_{\{\xi\in\Bn:\nabla f(\xi)=0\}}(-1)^{\ind_\xi\nabla f}=1-\sum_{\{\xi\in\dBn:\nabla f_0(\xi)=0,\,f_1(\xi)>0\}}(-1)^{\ind_\xi\nabla f_0}.$$
The latter quantity is non zero by assumptions, therefore the set of critical points of $f$ on $\Bn$, on which we are taking the first sum, cannot be empty.
\end{proof}\

\section{Appendix: Proof of Proposition \ref{coro-deriv}}\label{section-app-A}
By introducing a rotation in $\Bn$ and moving to $\R^n_+$ via $\inv$ we can give an expression for $\Gamma$ which is more convenient for our computation.
\begin{lemma}Let $A:\Bn\to\Bn$ be the rotation corresponding, via the $\inv$, to the translation of $T:x\to x+x_0$ on the half-place, namely $A=\inv\circ T\circ\inv^{-1}$. There holds:
\begin{equation*}\label{gamma-rot}
\Gamma(x_0,\l)=\L_n^\frac n2\int_{\R^n_+}\frac{\tilde K_A(\l y)}{\(\abs{\bar y}^2+(y_n+\D)^2-1\)^n}d\bar y dy_n-\L_n^\frac{n-1}2\b_n\D\int_{\de\R^n_+}\frac{\tilde H_A(\l\bar y)}{\(\abs{\bar y}^2+\D^2-1\)^{n-1}}d\bar y,
\end{equation*}
where $\tilde K_A=K\circ A\circ\inv,\tilde H_A=H\circ A\circ\inv$.
\end{lemma}
\begin{proof}
By doing a change of variables, we observe that
\begin{align*}
\Gamma(x_0,\l)=&\int_\Bn K_A(z)P_{x_0,\l}(Az)^\frac n2-\b_n\D\int_\dBn H_A(z)P_{x_0,\l}(Az)^\frac{n-1}2\\
=&\int_\Bn K_A(z)P_{0,\l}(z)^\frac n2-\b_n\D\int_\dBn H_A(z)P_{0,\l}(z)^\frac{n-1}2,
\end{align*}
Here we are using that $A$ is a rotation and its very definition, and we set $K_A=K\circ A,H_A=H\circ A$. Finally, changing variables twice and using the definitions in Section \ref{sec-preli}:
\begin{align*}
&\Gamma(x_0,\l)\\
=&\L_n^\frac n2\int_{\R^n_+}\frac{\tilde K_A(x)\l^n}{\(\abs{\bar x}^2+(x_n+\l\D)^2-\l^2\)^n}d\bar x dx_n+\L_n^\frac{n-1}2\b_n\D\int_{\de\R^n_+}\frac{\tilde H_A(\bar x)\l^{n-1}}{\(\abs{\bar x}^2+\l^2\(\D^2-1\)\)^{n-1}}d\bar x\\
=&\L_n^\frac n2\int_{\R^n_+}\frac{\tilde K_A(\l y)}{\(\abs{\bar y}^2+(y_n+\D)^2-1\)^n}d\bar y dy_n+\L_n^\frac{n-1}2\b_n\D\int_{\de\R^n_+}\frac{\tilde H_A(\l\bar y)}{\(\abs{\bar y}^2+\D^2-1\)^{n-1}}d\bar y
\end{align*}
\end{proof}

\begin{proof}[Proof of Proposition \ref{coro-deriv}]
We start by estimating the boundary term, where some cancellations occur due to symmetry. We expand $\tilde H_A(\l y)$ in $\l$ up to order $n-2$:
\begin{align*}
&\int_{\de\R^n_+}\frac{\tilde H_A(\l\bar y)}{\(\abs{\bar y}^2+\D^2-1\)^{n-1}}d\bar y\\
=&\tilde H_A(0)\int_{\de\R^n_+}\frac{d\bar y}{\(\abs{\bar y}^2+\D^2-1\)^{n-1}}+\sum_{1\le|\a|\le n-2}\frac{\l^{|\a|}}{|\a|!}\de_{\bar x_\a}\tilde H_A(0)\int_{\de\R^n_+}\frac{\bar y^\a}{\(\abs{\bar y}^2+\D^2-1\)^{n-1}}d\bar y\\
+&\underbrace{\int_{\de\R^n_+}\frac{\tilde H_A(\l\bar y)-\sum_{|\a|\le n-2}\frac{\l^{|\a|}}{|\a|!}\de_{\bar x_\a}\tilde H_A(0)\bar y^\a}{\(\abs{\bar y}^2+\D^2-1\)^{n-1}}d\bar y}_{=:I}\\
=&H(\xi)\int_{\de\R^n_+}\frac{d\bar y}{\(\abs{\bar y}^2+\D^2-1\)^{n-1}}+\l^2\frac1{4(n-1)}\Delta\tilde H_A(0)\int_{\de\R^n_+}\frac{|\bar y|^2}{\(\abs{\bar y}^2+\D^2-1\)^{n-1}}d\bar y+\dots\\
+&\l^{2\lfloor\frac{n-2}2\rfloor}\frac{(n-3)!!}{\(2\lfloor\frac{n-2}2\rfloor\)!\(2\lfloor\frac{n-2}2\rfloor\)!!\(n-3+2\lfloor\frac{n-2}2\rfloor\)!!}\Delta^{\lfloor\frac{n-2}2\rfloor}\tilde H_A(0)\int_{\de\R^n_+}\frac{|\bar y|^{2\lfloor\frac{n-2}2\rfloor}}{\(\abs{\bar y}^2+\D^2-1\)^{n-1}}d\bar y+I\\
=&\frac{H(\xi)}{\(\D^2-1\)^\frac{n-1}2}\int_{\de\R^n_+}\frac{d\bar y}{\(\abs{\bar y}^2+1\)^{n-1}}\\
+&\l^2\frac1{\(\D^2-1\)^\frac{n-3}2}\frac1{4(n-1)}\Delta\tilde H_A(0)\int_{\de\R^n_+}\frac{|\bar y|^2}{\(\abs{\bar y}^2+1\)^{n-1}}d\bar y+\dots\\
+&\l^{2\lfloor\frac{n-2}2\rfloor}\frac1{\(\D^2-1\)^\frac{n-2\lfloor\frac{n-2}2\rfloor}2-1}\frac{(n-3)!!}{\(2\lfloor\frac{n-2}2\rfloor\)!\(2\lfloor\frac{n-2}2\rfloor\)!!\(n-3+2\lfloor\frac{n-2}2\rfloor\)!!}\\
\times&\Delta^{\lfloor\frac{n-2}2\rfloor}\tilde H_A(0)\int_{\de\R^n_+}\frac{|\bar y|^{2\lfloor\frac{n-2}2\rfloor}}{\(\abs{\bar y}^2+1\)^{n-1}}d\bar y\\
+&I,
\end{align*}
where we used the formula
\begin{equation}
\label{laplapowers}\Delta^j|y|^{2j}=\frac{(2j)!!(n-3+2j)!!}{(n-3)!!}
\end{equation}
and the vanishing, due to symmetry, of integrals of homogeneous polynomials of odd degree or of degree $2j$ which are $j$-harmonic.\\
Moreover, in view of the conformal properties of the Laplacian, one has
\begin{equation}\label{laplaconf}
\Delta^j\tilde H_A(0)=(1+\xi_n)^{2j}\Delta^jH(\xi),
\end{equation}
hence the $j^\mathrm{th}$ term in the expansion equals
$$\l^{2j}\frac1{\(\D^2-1\)^\frac{n-2j-1}2}\frac{(n-3)!!}{(2j)!(2j)!!(n+2j-3)!!}(1+\xi_n)^{2j}\Delta^jH(\xi)\int_{\de\R^n_+}\frac{|\bar y|^{2j}}{\(\abs{\bar y}^2+1\)^{n-1}}d\bar y.$$

\

In the $j^\mathrm{th}$ order expansion, the remainder is actually $o\(\l^j\)$ because we get
\begin{align*}
&\int_{|\bar y|\le\frac1\l}\frac{\tilde H_A(\l\bar y)-\sum_{|\a|\le j}\frac{\l^{|\a|}}{|\a|!}\de_{\bar x_\a}\tilde H_A(0)\bar y^\a}{\(\abs{\bar y}^2+\D^2-1\)^{n-1}}d\bar y+\int_{|\bar y|>\frac1\l}\frac{\tilde H_A(\l\bar y)-\sum_{|\a|\le j}\frac{\l^{|\a|}}{|\a|!}\de_{\bar x_\a}\tilde H_A(0)\bar y^\a}{\(\abs{\bar y}^2+\D^2-1\)^{n-1}}d\bar y\\
=&\int_{|\bar y|\le\frac1\l}\frac{O\(|\l\bar y|^{j+1}\)}{\(\abs{\bar y}^2+\D^2-1\)^{n-1}}d\bar y+\int_{|\bar y|>\frac1\l}\frac{O\(|\l\bar y|^j\)}{\(\abs{\bar y}^2+\D^2-1\)^{n-1}}d\bar y\\
=&O\(\l^{j+1}\log\frac1\l\)+O\(\l^{n-1}\).
\end{align*}

\

In order to deal with higher order terms, we need another argument, since this would get non-converging integrals.\\
We split the cases $n$ even and $n$ odd.\\
If $n$ is even, the main order term in the denominator of $I$ is of odd order, hence its integral vanishes. Therefore,
\begin{align*}
I=&\l^{n-1}\int_{\de\R^n_+}\frac{\tilde H_A(\bar x)-\sum_{|\a|\le n-2}\frac1{|\a|!}\de_{\bar x_\a}\tilde H_A(0)\bar x^\a}{\(\abs{\bar x}^2+\l^2\(\D^2-1\)\)^{n-1}}d\bar x\\
=&\l^{n-1}\int_{\de\R^n_+}\frac{\tilde H_A(\bar x)-\sum_{|\a|\le n-2}\frac1{|\a|!}\de_{\bar x_\a}\tilde H_A(0)\bar x^\a-\frac1{(n-1)!}\sum_{|\a|=n-1}\de_{\bar x^\a}\tilde H_A(0)\bar x^\a\chi_{|x|\le1}}{\(\abs{\bar x}^2+\l^2\(\D^2-1\)\)^{n-1}}d\bar x\\
=&\l^{n-1}\int_{\de\R^n_+}\frac{\tilde H_A(\bar x)-\sum_{|\a|\le n-1}\frac1{|\a|!}\de_{\bar x_\a}\tilde H_A(0)\bar x^\a\chi_{|x|\le1}}{\abs{\bar x}^{2(n-1)}}d\bar x\\
+&\int_{\de\R^n_+}\(\tilde H_A(\bar x)-\sum_{|\a|\le n-1}\frac1{|\a|!}\de_{\bar x_\a}\tilde H_A(0)\bar x^\a\)\(\frac{\l^{n-1}}{\(\abs{\bar x}^2+\l^2\(\D^2-1\)\)^{n-1}}-\frac{\l^{n-1}}{|\bar x|^{2(n-1)}}\)d\bar x\\
+&O\(\sum_{j=0}^\frac{n-2}2\Delta^j\tilde H_A(0)\l^{n-1}\)\\
=&\l^{n-1}\int_{\de\R^n_+}\frac{\tilde H_A(\bar x)-\sum_{|\a|\le n-1}\frac1{|\a|!}\de_{\bar x_\a}\tilde H_A(0)\bar x^\a\chi_{|x|\le1}}{\abs{\bar x}^{2(n-1)}}d\bar x\\
+&\frac{\l^n}{n!}\sum_{|\a|=n}\de_{\bar x^\a}\tilde H_A(0)\int_{\de\R^n_+}\bar y^\a\(\frac1{\(\abs{\bar y}^2+\D^2-1\)^{n-1}}-\frac1{|\bar y|^{2(n-1)}}\)\\
+&\underbrace{\int_{\de\R^n_+}\(\tilde H_A(\l\bar y)-\sum_{|\a|\le n}\frac{\l^{|\a|}}{|\a|!}\de_{\bar x_\a}\tilde H_A(0)\bar y^\a\)\(\frac1{\(\abs{\bar y}^2+\D^2-1\)^{n-1}}-\frac1{|\bar y|^{2(n-1)}}\)d\bar y}_{=:I'}\\
+&o\(\sum_{j=0}^\frac{n-2}2\Delta^j\tilde H_A(0)\l^{2j}\)\\
=&\l^{n-1}\int_{\de\R^n_+}\frac{\tilde H_A(\bar x)-\sum_{|\a|\le n-1}\frac1{|\a|!}\de_{\bar x_\a}\tilde H_A(0)\bar x^\a\chi_{|x|\le1}}{\abs{\bar x}^{2(n-1)}}d\bar x\\
+&\l^n\(\D^2-1\)^\frac12\frac{(n-3)!!}{n!n!!(2n-3)!!}(\Delta)^\frac n2\tilde H_A(0)\int_{\de\R^n_+}|\bar y|^n\(\frac1{\(|\bar y|^2+1\)^{n-1}}-\frac1{|\bar y|^{2(n-1)}}\)d\bar y\\
+&I'+o\(\sum_{j=0}^\frac{n-2}2\Delta^j\tilde H_A(0)\l^{2j}\),
\end{align*}
where we used again \eqref{laplapowers}; one easily verifies that, due to the behaviors at $0$ at infinity, all the integrals are converging, hence everything is well defined.\\
After changing variables, the main terms are now
\begin{align*}
&\l^{n-1}(1+\xi_n)^{n-1}\int_{\dBn}\frac{H(z)-\sum_{|\a|\le n-1}\frac1{|\a|!}\de_{\xi_\a}H(\xi)(z-\xi)^\a}{|z-\xi|^{2(n-1)}}dz\\
+&\l^n\(\D^2-1\)^\frac12\frac{(n-3)!!}{n!n!!(2n-3)!!}(1+\xi_n)^n(\Delta)^\frac n2H(\xi)\int_{\de\R^n_+}|\bar y|^n\(\frac1{\(|\bar y|^2+1\)^{n-1}}-\frac1{|\bar y|^{2(n-1)}}\)d\bar y\\
+&I'+o\(\sum_{j=0}^\frac{n-2}2\Delta^jH(\xi)\l^{2j}\),
\end{align*}
where we used the fact that $\frac1{|\bar x|^2}=\frac{|z+\xi|^2}{|z-\xi|^2}$ and again \eqref{laplaconf}. The small $o$ term contains some new quantities arising when the terms of order $\l^{n-1}$ are transformed into each other.\\
The smallness of the remainder can be shown similarly as before, here and in the following.\\
Due to the asymptotic behavior of both factors, $I'$ can be dealt with similarly as $I$ and one can iterate the argument. In particular, using the series expansion
$$\frac1{\(|\bar y|^2+\D^2-1\)^{n-1}}=\sum_{j=0}^\infty(-1)^j\frac{(n+j-2)!}{j!(n-2)!}\frac{\(\D^2-1\)^j}{|\bar y|^{2(n+j-1)}},$$
we get, for any even $m>n$, the following $m^\mathrm{th}$ order term:
\begin{align*}
&\l^{m-1}(-1)^\frac{m-n}2(1+\xi_n)^{m-1}\frac{\(\frac{n+m}2-2\)!}{\(\frac{m-n}2\)!(n-2)!}\(\D^2-1\)^\frac{m-n}2\\
\times&\int_{\dBn}\(H(z)-\sum_{|\a|\le m-1}\frac1{|\a|!}\de_{\xi_\a}H(\xi)(z-\xi)^\a\)\frac{|z+\xi|^{m-n}}{|z-\xi|^{n+m-2}}dz\\
+&\l^m\(\D^2-1\)^\frac{m-n+1}2\frac{(n-3)!!}{m!m!!(n+m-3)!!}(1+\xi_n)^m(\Delta)^\frac m2H(\xi)\\
\times&\int_{\de\R^n_+}|\bar y|^m\(\frac1{\(|\bar y|^2+1\)^{n-1}}-\sum_{j=0}^\frac{m-n}2(-1)^j\frac{(n+j-2)!}{j!(n-2)!}\frac1{|\bar y|^{2(n+j-1)}}\)d\bar y\\
+&\int_{\de\R^n_+}\(\tilde H_A(\l\bar y)-\sum_{|\a|\le m}\frac{\l^{|\a|}}{|\a|!}\de_{\bar x_\a}\tilde H_A(0)\bar y^\a\)\\
\times&\(\frac1{\(\abs{\bar y}^2+\D^2-1\)^{n-1}}-\sum_{j=0}^\frac{m-n}2(-1)^j\frac{(n+j-2)!}{j!(n-2)!}\frac{\(\D^2-1\)^j}{|\bar y|^{2(n+j-1)}}\)d\bar y\\
+&o\(\sum_{j=0}^\frac{m-2}2\Delta^jH(\xi)\l^{2j}\).
\end{align*}
In particular, we point out that if $n=2$ this is the main order term in the boundary estimates, and it equals
\begin{equation}\label{fraclap}
-\l\pi(1+\xi_n)(-\Delta)^\frac12H(\xi).
\end{equation}\

Let us now consider the case $n$ odd. Here, the first term does not vanish and it gives rise to a logarithmic term. In fact,
\begin{align*}
I=&\l^{n-1}\int_{\de\R^n_+}\frac{\tilde H_A(\bar x)-\sum_{|\a|\le n-2}\frac1{|\a|!}\de_{\bar x_\a}\tilde H_A(0)\bar x^\a}{\(\abs{\bar x}^2+\l^2\(\D^2-1\)\)^{n-1}}d\bar x\\
=&\frac{\l^{n-1}}{(n-1)!}\sum_{|\a|=n-1}\de_{\bar x^\a}\tilde H_A(0)\int_{|\bar x|\le1}\frac{\bar x^\a}{\(\abs{\bar x}^2+\l^2\(\D^2-1\)\)^{n-1}}d\bar x\\
+&\l^{n-1}\int_{\de\R^n_+}\frac{\tilde H_A(\bar x)-\sum_{|\a|\le n-2}\frac1{|\a|!}\de_{\bar x_\a}\tilde H_A(0)\bar x^\a-\frac1{(n-1)!}\sum_{|\a|=n-1}\de_{\bar x^\a}\tilde H_A(0)\bar x^\a\chi_{|x|\le1}}{\(\abs{\bar x}^2+\l^2\(\D^2-1\)\)^{n-1}}d\bar x\\
=&\l^{n-1}\(\log\frac1\l+O(1)\)\frac{(n-3)!!}{(n-1)!(n-1)!!(2n-4)!!}\Delta^\frac{n-1}2\tilde H_A(0)\omega_{n-2}\\
+&\l^{n-1}\int_{\de\R^n_+}\frac{\tilde H_A(\bar x)-\sum_{|\a|\le n-1}\frac1{|\a|!}\de_{\bar x_\a}\tilde H_A(0)\bar x^\a\chi_{|x|\le1}}{\(\abs{\bar x}^2+\l^2\(\D^2-1\)\)^{n-1}}d\bar x+O\(\l^{n-1}\sum_{j=0}^\frac{n-3}2\Delta^j\tilde H_A(0)\)\\
=&\l^{n-1}\log\frac1\l\frac{(n-3)!!}{(n-1)!(n-1)!!(2n-4)!!}\Delta^\frac{n-1}2\tilde H_A(0)\omega_{n-2}(1+o(1))\\
+&\l^{n-1}\int_{\de\R^n_+}\frac{\tilde H_A(\bar x)-\sum_{|\a|\le n-1}\frac1{|\a|!}\de_{\bar x_\a}\tilde H_A(0)\bar x^\a\chi_{|x|\le1}}{\abs{\bar x}^{2(n-1)}}d\bar x\\
-&\int_{\de\R^n_+}\(\tilde H_A(\bar x)-\sum_{|\a|\le n-1}\frac1{|\a|!}\de_{\bar x_\a}\tilde H_A(0)\bar x^\a\)\(\frac{\l^{n-1}}{|\bar x|^{2(n-1)}}-\frac{\l^{n-1}}{\(\abs{\bar x}^2+\l^2\(\D^2-1\)\)^{n-1}}\)d\bar x\\
+&o\(\sum_{j=0}^\frac{n-3}2\Delta^j\tilde H_A(0)\l^{2j}\)\\
=&\l^{n-1}\log\frac1\l\frac{(n-3)!!}{(n-1)!(n-1)!!(2n-4)!!}(1+\xi_n)^{n-1}\Delta^\frac{n-1}2H(\xi)\omega_{n-2}(1+o(1))\\
+&\l^{n-1}(1+\xi_n)^{n-1}\int_{\dBn}\frac{H(z)-\sum_{|\a|\le n-1}\frac1{|\a|!}\de_{\xi_\a}H(\xi)(z-\xi)^\a}{|z-\xi|^{2(n-1)}}dz\\
-&\int_{\de\R^n_+}\(\tilde H_A(\l\bar y)-\sum_{|\a|\le n}\frac{\l^{|\a|}}{|\a|!}\de_{\bar x_\a}\tilde H_A(0)\bar y^\a\)\(\frac1{|\bar y|^{2(n-1)}}-\frac1{\(\abs{\bar y}^2+\D^2-1\)^{n-1}}\)d\bar y\\
+&o\(\sum_{j=0}^\frac{n-3}2\Delta^j H_\xi(0)\l^{2j}\);
\end{align*}
iterating, for any odd $m>n$ we get
\begin{align*}
&\l^{m-1}\log\frac1\l(-1)^\frac{m-n}2\frac{(n-3)!!}{(m-1)!(m-1)!!(n+m-4)!!}\frac{\(\frac{n+m}2-2\)!}{\(\frac{m-n}2\)!(n-2)!}\(\D^2-1\)^\frac{m-n}2\\
\times&(1+\xi_n)^{m-1}\Delta^\frac{m-1}2H(\xi)\omega_{n-2}(1+o(1))\\
+&\l^{m-1}(-1)^\frac{m-n}2(1+\xi_n)^{m-1}\frac{\(\frac{n+m}2-2\)!}{\(\frac{m-n}2\)!(n-2)!}\(\D^2-1\)^\frac{m-n}2\\
\times&\int_{\dBn}\(H(z)-\sum_{|\a|\le m-1}\frac1{|\a|!}\de_{\xi_\a}H(\xi)(z-\xi)^\a\)\frac{|z+\xi|^{m-n}}{|z-\xi|^{n+m-2}}dz\\
+&\int_{\de\R^n_+}\(\tilde H_A(\l\bar y)-\sum_{|\a|\le m}\frac{\l^{|\a|}}{|\a|!}\de_{\bar x_\a}\tilde H_A(0)\bar y^\a\)\\
\times&\(\frac1{\(\abs{\bar y}^2+\D^2-1\)^{n-1}}-\sum_{j=0}^\frac{m-n}2(-1)^j\frac{(n+j-2)!}{j!(n-2)!}\frac{\(\D^2-1\)^j}{|\bar y|^{2(n+j-1)}}\)d\bar y\\
+&o\(\sum_{j=0}^\frac{m-2}2\Delta^jH(\xi)\l^{2j}\).
\end{align*}\

The argument to estimate the interior terms is similar. We expand $\tilde K(\l y)$ up to order $n-1$, which is the highest power that can be integrated against $P_{x_0,\l}^\frac n2$.
\begin{align*}
&\int_{\R^n_+}\frac{\tilde K_A(\l y)}{\(\abs{\bar y}^2+(y_n+\D)^2-1\)^n}d\bar y dy_n\\
=&\tilde K_A(0)\int_{\R^n_+}\frac{d\bar ydy_n}{\(\abs{\bar y}^2+(y_n+\D)^2-1\)^n}\\
+&\sum_{1\le|\a|\le n-1}\frac{\l^{|\a|}}{|\a|!}\de_{x_\a}\tilde K_A(0)\int_{\de\R^n_+}\frac{y^\a}{\(\abs{\bar y}^2+(y_n+\D)^2-1\)^n}d\bar ydy_n\\
+&\underbrace{\int_{\R^n_+}\frac{\tilde K_A(\l y)-\sum_{|\a|\le n-1}\frac{\l^n}{|\a|}\de_{x_\a}\tilde K_A(0)y^\a}{\(\abs{\bar y}^2+(y_n+\D)^2-1\)^n}d\bar y dy_n}_{=:I''}\\
=&K(\xi)\int_{\R^n_+}\frac{d\bar ydy_n}{\(\abs{\bar y}^2+(y_n+\D)^2-1\)^n}+\l\de_{x_n}\tilde K_A(\xi)\int_{\R^n_+}\frac{y_n}{\(\abs{\bar y}^2+(y_n+\D)^2-1\)^n}d\bar ydy_n\\
+&\sum_{j=2}^{n-1}\l^j\sum_{i=0}^{\lfloor\frac j2\rfloor}\frac1{(j-2i)!(2i)!}\frac{(n-3)!!}{(2i)!!(n+2i-3)!!}\de_{x_n}^{j-2i}\Delta_{\bar x}^i\tilde K_A(0)\int_{\R^n_+}\frac{|\bar y|^{2i}y_n^{j-2i}}{\(\abs{\bar y}^2+(y_n+\D)^2-1\)^n}d\bar ydy_n\\
+&I''\\
=&K(\xi)\int_{\R^n_+}\frac{d\bar ydy_n}{\(\abs{\bar y}^2+(y_n+\D)^2-1\)^n}-\l(1+\xi_n)\de_\nu K(\xi)\int_{\R^n_+}\frac{y_n}{\(\abs{\bar y}^2+(y_n+\D)^2-1\)^n}d\bar ydy_n\\
+&\sum_{j=2}^{n-1}\l^j\sum_{i=0}^{\lfloor\frac j2\rfloor}\frac1{(j-2i)!(2i)!}\frac{(n-3)!!}{(2i)!!(n+2i-3)!!}(-1)^j(1+\xi_n)^j\de_\nu^{j-2i}\Delta_\tau^iK(\xi)\\
\times&\int_{\R^n_+}\frac{|\bar y|^{2i}y_n^{j-2i}}{\(\abs{\bar y}^2+(y_n+\D)^2-1\)^n}d\bar ydy_n\\
+&I'',
\end{align*}
where we wrote the derivation in $\bar x$ and $x_n$ as
$$\sum_{|\a|=j}\de_{x_\a}=\sum_{i=0}^j\frac{j!}{(j-i)!i!}\sum_{|\b|=i}\de_{x_n}^{j-i}\de_{\bar x_\b}$$
and used again cancellation by symmetry and \eqref{laplapowers}. In the last step, we used \eqref{laplaconf} (in $\bar x$) and that
$$\de_{x_n}^j=(-1)^j(1+\xi_n)^j\de_\nu.$$
In the case $n=2$, since $\int_{\R^2_+}\frac{y_2}{\(\bar y^2+(y_n+\D)^2-1\)^n}d\bar ydy_2=\frac\pi2\(\sqrt{\D^2-1}-\D\)$, putting together with \eqref{fraclap} we get the first order expansion
$$\Gamma(x_0,\l)=\psi(\xi)-2\pi(1+\xi_n)\l\(\(\D-\sqrt{\D^2-1}\)\,\de_\nu K(\xi)-2\D(-\Delta)^\frac12H(\xi)\)+o(\l),$$
whereas when $n\ge3$ the first order expansion contains only the interior term:
$$\Gamma(x_0,\l)=\psi(\xi)-a_{n,0,1}(1+\xi_n)\l\de_\nu K(\xi).$$
As for $I''$, we get \emph{local} terms involving derivatives of $\tilde K_A$ and \emph{non-local} terms similarly as before:
\begin{align*}
I''=&\l^n\int_{\R^n_+}\frac{\tilde K_A(x)-\sum_{|\a|\le n-1}\frac1{|\a|!}\de_{x_\a}\tilde K_A(0)x^\a}{\(\abs{\bar x}^2+(x_n+\l\D)^2-\l^2\)^n}d\bar x dx_n\\
=&\frac{\l^n}{n!}\sum_{|\a|=n}\de_{x_\a}\tilde K_A(0)\int_{|x|\le1}\frac{x^\a}{\(\abs{\bar x}^2+(x_n+\l\D)^2-\l^2\)^n}d\bar xdx_n\\
+&\l^n\int_{\R^n_+}\frac{\tilde K_A(x)-\sum_{|\a|\le n-1}\frac1{|\a|!}\de_{x_\a}\tilde K_A(0)x^\a-\frac1{n!}\sum_{|\a|=n}\de_{x_\a}\tilde K_A(0)x^\a\chi_{|x|\le1}}{\(\abs{\bar x}^2+(x_n+\l\D)^2-\l^2\)^n}d\bar x dx_n\\
=&\l^n\(\log\frac1\l+O(1)\)\sum_{i=0}^{\lfloor\frac n2\rfloor}\frac{(n-3)!!}{(n-2i)!(2i)!(2i)!!(n+2i-3)!!}\de_{x_n}^{n-2i}\Delta_{\bar x}^i\tilde K_A(0)\omega_{n-2}\int_0^\pi\sin^{n-2i}tdt\\
+&\l^n\int_{\R^n_+}\frac{\tilde K_A(x)-\sum_{|\a|\le n}\frac1{|\a|!}\de_{x_\a}\tilde K_A(0)x^\a\chi_{|x|\le1}}{\(\abs{\bar x}^2+(x_n+\l\D)^2-\l^2\)^n}d\bar x dx_n\\
+&O\(\l^n\sum_{j=0}^{n-1}\sum_{i=0}^{\lfloor\frac j2\rfloor}\de_{x_n}^{j-2i}\Delta_{\bar x}^i\tilde K_A(0)\)\\
=&\l^n\log\frac1\l\sum_{i=0}^{\lfloor\frac n2\rfloor}\frac{(n-3)!!}{(n-2i)!(2i)!(2i)!!(n+2i-3)!!}\de_{x_n}^{n-2i}\Delta_{\bar x}^i\tilde K_A(0)\omega_{n-2}\int_0^\pi\sin^{n-2i}tdt(1+o(1))\\
+&\l^n\int_{\R^n_+}\frac{\tilde K_A(x)-\sum_{|\a|\le n}\frac1{|\a|!}\de_{x_\a}\tilde K_A(0)x^\a\chi_{|x|\le1}}{\abs{\bar x}^{2n}}d\bar x dx_n\\
+&\underbrace{\int_{\R^n_+}\(\tilde K_A(\l y)-\sum_{|\a|\le n}\frac{\l^{|\a|}}{|\a|!}\de_{x_\a}\tilde K_A(0)y^\a\)\(\frac1{\(|\bar y|^2+\(y_n+\D\)^2-1\)^n}-\frac1{|y|^{2n}}\)d\bar ydy_n}_{=:I'''}\\
+&o\(\sum_{j=0}^{n-1}\l^j\sum_{i=0}^{\lfloor\frac j2\rfloor}\de_{x_n}^{j-2i}\Delta_{\bar x}^i\tilde K_A(0)\)\\
=&\l^n\log\frac1\l\sum_{i=0}^{\lfloor\frac n2\rfloor}\frac{(n-3)!!}{(n-2i)!(2i)!(2i)!!(n+2i-3)!!}(-1)^n(1+\xi_n)^n\de_\nu^{n-2i}\Delta_\tau^i K(\xi)\omega_{n-2}\\
\times&\int_0^\pi\sin^{n-2i}tdt(1+o(1))\\
+&\l^n(1+\xi_n)^n\int_\Bn\frac{K(z)-\sum_{|\a|\le n}\frac1{|\a|!}\de_{\xi_\a}K(\xi)(z-\xi)^\a}{|z-\xi|^{2n}}dz\\
+&I''+o\(\sum_{j=0}^{n-1}\l^j\sum_{i=0}^{\lfloor\frac j2\rfloor}\de_\nu^{j-2i}\Delta_\tau x^iK(\xi)\).
\end{align*}
In order to iterate and find the next order terms, we again need a series expansion: we get
$$\frac1{\(|\bar y|^2+\(y_n+\D\)^2-1\)^n}=\sum_{j=0}^\infty\frac1{|y|^{2(n+j)}}\sum_{i=0}^{\lfloor\frac j2\rfloor}(-1)^{j-i}\frac{(n+j-i-1)!}{(n-1)!i!(j-2i)!}\(\D^2-1\)^i\(2\D\)^{j-2i}|y|^{2i}y_n^{j-2i},$$
which in turn comes from
$$\frac1{at^2+bt+1}=\sum_{j=0}^\infty t^j\sum_{i=0}^{\lfloor\frac j2\rfloor}(-1)^{j-i}\frac{(n+j-i-1)!}{(n-1)!i!(j-2i)!}a^ib^{j-2i}.$$
Therefore, for $m>n$ we get:
\begin{align*}
&\l^m\log\frac1\l\sum_{i=0}^{\lfloor\frac{m-n}2\rfloor}(-1)^{m-n-i}\frac{(m-i-1)!}{(n-1)!i!(m-n-2i)!}\(\D^2-1\)^i\(2\D\)^{m-n-2i}\\
\times&\sum_{j=0}^{\lfloor\frac m2\rfloor}\frac{(n-3)!!}{(m-2j)!(2j)!(2j)!!(n+2j-3)!!}(-1)^m(1+\xi_n)^m\de_\nu^{m-2j}\Delta_\tau^j K(\xi)\omega_{n-2}\\
\times&\int_0^\pi\sin^{2m-n-2i-2j}tdt(1+o(1))\\
+&\l^m(1+\xi_n)^m\sum_{i=0}^{\lfloor\frac{m-n}2\rfloor}(-1)^{m-n-i}\frac{(m-i-1)!}{(n-1)!i!(m-n-2i)!}\(\D^2-1\)^i\(2\D\)^{m-n-2i}\\
\times&\int_\Bn\(K(z)-\sum_{|\a|\le m}\frac1{|\a|!}\de_{\xi_\a}K(\xi)(z-\xi)^\a\)\frac{|z+\xi|^{2i}\(1-|z|^2\)^{m-n-2i}}{|z-\xi|^{2m}}dz\\
+&\int_{\R^n_+}\(\tilde K_A(\l y)-\sum_{|\a|\le m}\frac{\l^{|\a|}}{|\a|!}\de_{x_\a}\tilde K_A(0)y^\a\)\(\frac1{\(|\bar y|^2+\(y_n+\D\)^2-1\)^n}\right.\\
-&\left.\sum_{j=0}^{m-n}\frac1{|y|^{2(n+j)}}\sum_{i=0}^{\lfloor\frac j2\rfloor}(-1)^{j-i}\frac{(n+j-i-1)!}{(n-1)!i!(j-2i)!}\(\D^2-1\)^i\(2\D\)^{j-2i}|y|^{2i}y_n^{j-2i}\)d\bar ydy_n\\
+&o\(\sum_{j=0}^{m-1}\l^j\sum_{i=0}^{\lfloor\frac j2\rfloor}\de_\nu^{j-2i}\Delta_\tau K(\xi)\).
\end{align*}
The proof is now complete, since all the quantities are the same as in Definition \ref{def-constants}.
\end{proof}

\end{document}